\documentclass[a4paper, reqno]{amsart}

\usepackage[all]{xy}
\usepackage[english]{babel}
\usepackage[utf8]{inputenc}
\usepackage{amssymb,amsmath,amsthm}
\usepackage{amsfonts}
\usepackage{graphicx}
\usepackage{psfrag}
\usepackage[dvipsnames]{xcolor}
\usepackage[hidelinks]{hyperref}
\usepackage{csquotes}
\usepackage{amsrefs}

\allowdisplaybreaks

\newcommand{\id}{\mathrm{id}}
\newcommand{\Ric}{\mathrm{Ric}}

\newcommand{\e}{\epsilon}

\renewcommand{\S}{\Sigma}

\newcommand{\n}{\nabla}
\newcommand{\on}{\bar \nabla}
\newcommand{\en}{\on^E}
\newcommand{\im}{\mathrm{im}}
\renewcommand{\L}{\mathcal{L}}

\newcommand{\bR}{{\bar R}}
\newcommand{\g}{{\bar g}}
\renewcommand{\o}{\omega}
\renewcommand{\O}{\mathcal O}
\newcommand{\U}{\mathcal U}
\newcommand{\supp}{\mathrm{supp}}
\newcommand{\tr}{\mathrm{tr}}
\renewcommand{\a}{\alpha}

\renewcommand{\b}{\beta}
\renewcommand{\d}{\partial}
\newcommand{\Hess}{\mathrm{Hess}}
\newcommand{\abs}[1]{\left\lvert#1\right\rvert}
\newcommand{\norm}[1]{\left\lVert#1\right\rVert}
\newcommand{\ldr}[1]{\left\langle #1 \right\rangle}
\renewcommand{\div}{\mathrm{div}}
\newcommand{\grad}{\mathrm{grad}}
\renewcommand{\H}{\mathcal H}

\renewcommand{\Re}{\mathrm{Re}}

\newcommand{\ext}{\mathrm{ext}}
\newcommand{\bBox}{{\bar \Box}}

\newcommand{\f}{\phi}
\newcommand{\W}{Z}

\theoremstyle{plain}
\newtheorem{thm}{Theorem}[section]
\newtheorem{prop}[thm]{Proposition}
\newtheorem{lemma}[thm]{Lemma}
\newtheorem{cor}[thm]{Corollary}
\theoremstyle{definition}
\newtheorem{definition}[thm]{Definition}
\newtheorem{notation}[thm]{Notation}

\newtheorem{remark}[thm]{Remark}
\newtheorem{example}[thm]{Example}

\newtheorem{assumption}[thm]{Assumption}

\newcommand{\R}[0]{\mathbb{R}}							
\newcommand{\N}[0]{\mathbb{N}}							
\newcommand{\Z}[0]{\mathbb{Z}}							

\author{Oliver Lindblad Petersen}
\title[Extension of Killing vector fields]{Extension of Killing vector fields beyond compact Cauchy horizons}
\address{Department of Mathematics, Stanford University, CA 94305-2125, USA}
\email{oliverlp@stanford.edu}

\begin{document}
\hbadness=100000
\vbadness=100000

\begin{abstract}

We prove that any compact Cauchy horizon with constant non-zero surface gravity in a smooth vacuum spacetime is a smooth Killing horizon.
The novelty here is that the Killing vector field is shown to exist on \emph{both} sides of the horizon.
This generalises classical results by Moncrief and Isenberg, by dropping the assumption that the metric is analytic.
In previous work by Rácz and the author, the Killing vector field was constructed on the globally hyperbolic side of the horizon.
In this paper, we prove a new unique continuation theorem for wave equations through smooth compact lightlike (characteristic) hypersurfaces which allows us to extend the Killing vector field beyond the horizon.
The main ingredient in the proof of this theorem is a novel Carleman type estimate.
Using a well-known construction, our result applies in particular to smooth stationary asymptotically flat vacuum black hole spacetimes with event horizons with constant non-zero surface gravity.
As a special case, we therefore recover Hawking's local rigidity theorem for such black holes, which was recently proven by Alexakis-Ionescu-Klainerman using a different Carleman type estimate.

\end{abstract}

\maketitle

\tableofcontents
\begin{sloppypar}

\section{Introduction}

A classical conjecture in General Relativity, by Moncrief and Isenberg \cite{MoncriefIsenberg1983}, states that any compact Cauchy horizon in a vacuum spacetime is a Killing horizon. 
It says in particular that vacuum spacetimes containing compact Cauchy horizons admit a Killing vector field and are therefore non-generic.
One could therefore consider this as a first step towards Penrose's strong cosmic censorship conjecture in general relativity, without symmetry assumptions.
Indeed, it would imply that maximal globally hyperbolic vacuum developments of generic initial data cannot be extended over compact Cauchy horizons (see \cite{Petersen2018} for a more precise explanation of this).
The conjecture also turns out to be a natural generalisation of Hawking's local rigidity theorem for stationary vacuum black holes.
Moncrief and Isenberg have made remarkable progress on their conjecture in the last decades, see \cite{MoncriefIsenberg1983}, \cite{IsenbergMoncrief1985}, \cite{MoncriefIsenberg2008} and \cite{MoncriefIsenberg2018}, under the assumption that the spacetime metric is \emph{analytic}.

\begin{figure*}
  \begin{center}
    \includegraphics[scale = 0.5]{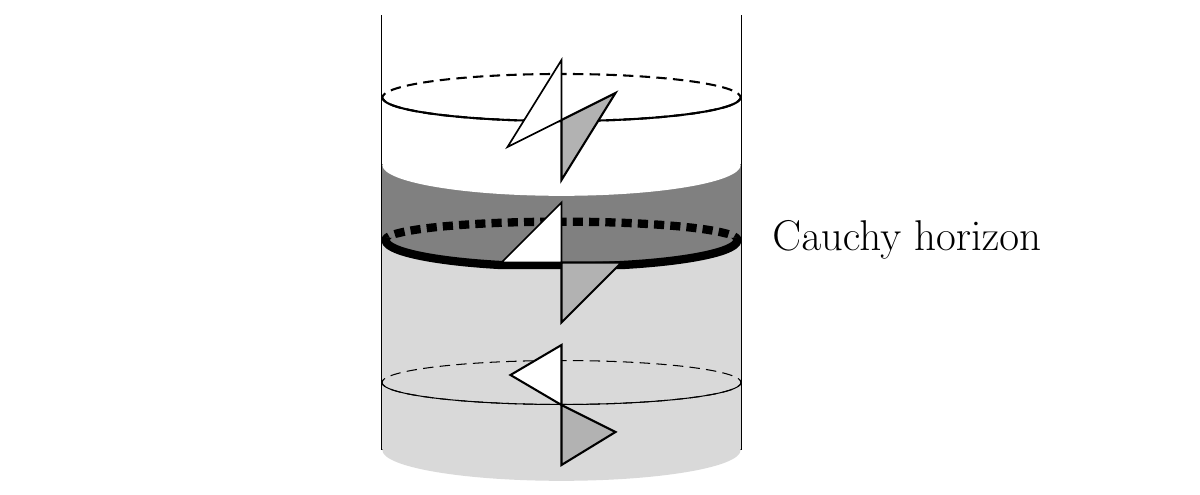}
  \end{center}
  \caption{
  	The lightly shaded region illustrates the globally hyperbolic region, where the Killing vector field is known to exist. 
  	We prove in this paper that the Killing vector field extends beyond the Cauchy horizon into the darkly shaded region.
  	}
  	\label{fig: Regions_CauchyHorizon}
\end{figure*}

In this paper, we are interested in the case when the spacetime metric is only assumed to be \emph{smooth}, as opposed to analytic.
The main problem in the smooth setting is that we do not have the Cauchy-Kowalevski theorem at our disposal anymore.
We instead need to propagate the Killing vector field using linear wave equations.
Since Cauchy horizons are lightlike hypersurfaces (c.f.\ Figure \ref{fig: Regions_CauchyHorizon}), the metric degenerates and we thus need to perform a singular analysis of wave equations close to the horizon.
The purpose of this paper is to present methods that replace the Cauchy-Kowalevski theorem in proving Moncrief-Isenberg's conjecture, assuming the surface gravity can be normalised to a non-zero constant.
This allows us to drop the highly restrictive assumption that the spacetime metric is analytic.

The first generalisation of the Moncrief-Isenberg results to smooth metrics was done by Friedrich-Rácz-Wald in \cite{FRW1999}.
They showed that if the surface gravity is a non-zero constant and the generators (the lightlike curves) of the horizon are \emph{all closed}, then there exists a Killing vector field on the \emph{globally hyperbolic side} of the Cauchy horizon.
The proof relies on a clever transform of the problem into a characteristic Cauchy problem, with initial data prescribed on two intersecting lightlike hypersurfaces.
This initial value problem can be solved using classical results, see for example \cite{Rendall1990}.

If the generators do not all close, one cannot use the approach of Friedrich-Rácz-Wald.
Due to this, the author developed new methods to solve linear wave equations with initial data on compact Cauchy horizons with constant non-zero surface gravity, see \cite{Petersen2018}.
Using \cite{Petersen2018}*{Thm. 1.6}, Rácz and the author generalised the result of Friedrich-Rácz-Wald by dropping the assumption that the generators close.
We proved that if the surface gravity is a non-zero constant, then there always exists a Killing vector field on the globally hyperbolic side of the Cauchy horizon, see \cite{PetersenRacz2018}*{Thm. 1.2}.
It is worth noting that our result allows \enquote{ergodic} behaviour of the generators, a case which was open even for analytic spacetime metrics.

However, the results in \cite{FRW1999} and \cite{PetersenRacz2018} do  not prove that the Cauchy horizon is a Killing horizon.
The Killing vector field was in both papers only shown to exist on the globally hyperbolic side of the Cauchy horizon.
It remains to prove that the Killing vector field extends beyond the horizon.
The difficulty here is that beyond the Cauchy horizon there are \emph{closed causal curves}, which makes the classical theory of wave equations useless. 
This is illustrated in Figure \ref{fig: Regions_CauchyHorizon} (see also Example \ref{ex: Misner}).
The main result of this paper is a solution to this problem.
We prove that if the surface gravity of the compact Cauchy horizon is a non-zero constant, then the Killing vector field constructed in \cite{PetersenRacz2018}*{Thm.\ 1.2} can indeed be extended beyond the Cauchy horizon, see \textbf{Theorem \ref{thm: Killing Horizon main}} below.
For the definitions and precise results, we refer to Subsection \ref{subsec: main results}.

Our argument is based on a new type of ``non-local" unique continuation theorem for wave equations through smooth compact lightlike (characteristic) hypersurfaces.
We prove that if a solution to a linear wave equation vanishes to infinite order \emph{everywhere} along a smooth \emph{compact} lightlike hypersurface, with constant non-zero surface gravity, in a spacetime satisfying the dominant energy condition, then the solution vanishes on an open neighbourhood of the hypersurface.
This is the main analytical novelty of this paper, see \textbf{Theorem \ref{thm: Wave main}} and the stronger, yet more technical, \textbf{Theorem \ref{thm: unique continuation}} below.
In order to extend the Killing vector field, using our unique continuation result, we apply an important recent result by Ionescu-Klainerman \cite{IonescuKlainerman2013}*{Prop. 2.10}.

Our result is the first unique continuation theorem for wave equations through smooth lightlike (characteristic) hypersurfaces, apart from our \cite{Petersen2018}*{Cor.\ 1.8}, which is a one-sided version of the result here.
Indeed, \enquote{local} unique continuation, in the spirit of Hörmander's classical theorem for pseudo-convex hypersurfaces \cite{Hormander1985}, is false for smooth lightlike hypersurfaces.
As it turns out, not only the assumption of compactness of the lightlike hypersurfaces is important, also our assumption on the surface gravity is crucial, see Remark \ref{rmk: pseudo-convex}.
Example \ref{ex: van surf grav} shows that unique continuation is false for general compact lightlike hypersurfaces with vanishing surface gravity.

\begin{figure*} 
    \includegraphics[scale = 0.4]{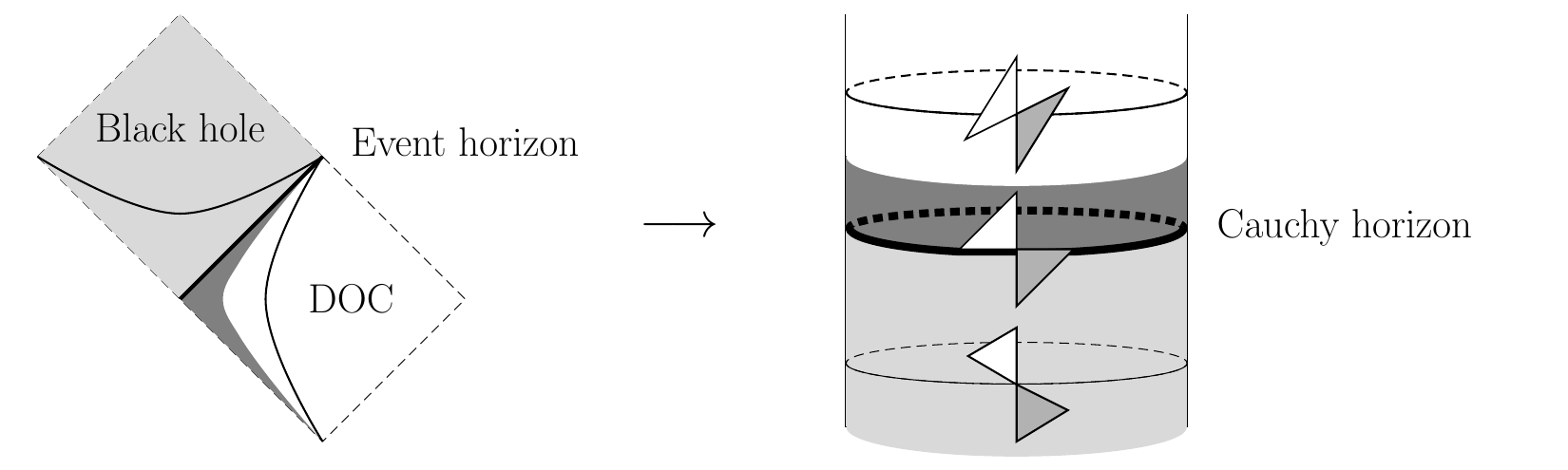}
  \caption{
  	Our extension of the Killing vector field \emph{beyond} the compact Cauchy horizon lifts to a Killing vector field in the \emph{domain of outer communication} (close to the event horizon) of the stationary black hole spacetime, reproving a recent theorem by Alexakis-Ionescu-Klainerman, without using a bifurcation surface.
  	}
  	\label{fig: EH to CH}
\end{figure*}

In order to explain how this work is related to the black hole uniqueness conjecture in general relativity, let us first recall the formulation and the state of art of that conjecture.
It asserts that the domain of outer communication of any $4$-dimensional stationary asymptotically flat vacuum black hole spacetime is isometric to the domain of outer communication of a Kerr spacetime.
By classical work by Carter \cite{C1971} and Robinson \cite{R1975}, the conjecture is proven under the additional assumption of non-degeneracy of the event horizon and \emph{axisymmetry} of the spacetime.
Using these results, Hawking proved that the non-extremal Kerr spacetimes are the only \emph{analytic} stationary asymptotically flat vacuum black hole spacetimes with non-degenerate event horizons, see \cite{Hawking1972}, \cite{HawkingEllis1973}, \cite{ChruscielCosta2008}.
He showed that the event horizon necessarily is a Killing horizon with the corresponding Killing vector field defined on the entire domain of outer communication, showing that the spacetime is axisymmetric.
Hawking's proof heavily relies on the assumption that the spacetime metric is analytic and does not extend to smooth spacetimes.

Alexakis, Ionescu and Klainerman have proven that \emph{smooth} stationary asymptotically flat vacuum black holes with bifurcate event horizons are Killing horizons, i.e.\ there exists a Killing vector field in an open neighbourhood of the event horizon, see \cite{AIK2010}*{Thm. 1.1} (applied to stationary black holes) and the related \cite{IonescuKlainerman2013}.
Their result generalises Hawking's result to \emph{smooth} (as opposed to analytic) spacetime metrics and is therefore referred to as \emph{Hawking's local rigidity without analyticity}.
Our main result can in fact be applied to reprove Hawking's local rigidity for \emph{smooth} stationary asymptotically flat vacuum black hole spacetimes, with event horizons with non-zero constant surface gravity, c.f.\ Figure \ref{fig: EH to CH}.
(Recall that bifurcate event horizons automatically have constant non-zero surface gravity, \cite{IK2009}*{p. 38}.)
As a special case of our result, we therefore get an alternative proof of the result by Alexakis-Ionescu-Klainerman, see \textbf{Theorem \ref{thm: Killing Extension bh}} below.
The main difference is that our proof does \emph{not} rely on the existence of a bifurcation surface. 
We extend the Killing vector field from either the future \emph{or} the past event horizon, not from both.

Let us remark that the result by Alexakis-Ionescu-Klainerman cannot be applied to prove our Theorem \ref{thm: Killing Horizon main} (that compact Cauchy horizons with constant non-zero surface gravity in vacuum spacetimes are Killing horizons).
The reason is that it is not known (in fact, it is a highly non-trivial open question) whether any such compact Cauchy horizon can be lifted to the future or past part of a bifurcate lightlike hypersurface in a covering vacuum spacetime. 
We avoid this issue by proving the unique continuation statement directly for compact Cauchy horizons.

Note that neither our result nor the result by Alexakis-Ionescu-Klainerman proves that the Killing vector field extends to the full domain of outer communication, in general.
Only for small perturbations of the Kerr spacetimes have Alexakis-Ionescu-Klainerman proven that the Killing vector field extends to the full domain of outer communication, see \cite{AIK2010_2}, \cite{AIK2014} and the related \cite{IK2009}, \cite{IK2009_2}.

Before we proceed by presenting the precise formulation of the main results, let us remark that all known examples of compact Cauchy horizons in vacuum spacetimes have constant non-zero surface gravity.
It is conceivable (and widely believed) that this is the case for any compact Cauchy horizon in a vacuum spacetime, see \cite{HIW2007} and \cite{RB2021} for partial progress on this problem. 
This is however still a rather subtle open question. 
In case the spacetime metric is analytic on the other hand, Moncrief and Isenberg have shown in their series of works that the surface gravity can, under general assumptions, be normalised to a constant.
In some special cases, they have even been able to prove that this constant must indeed be non-zero.

\subsection{Main results} \label{subsec: main results}

Let $M$ be a spacetime, i.e.\ a time-oriented connected Lorentzian manifold, of dimension $n+1 \geq 2$.
Let $\S$ denote a closed acausal topological hypersurface in $M$.
We assume that $\S$ has no boundary, but we do not assume $\S$ to be compact.
It can then be shown that $D(\S) \subset M$ is an open globally hyperbolic submanifold, with Cauchy hypersurface $\S$, and
\[
	\partial D(\S) = \H_- \cup \H_+,
\]
where $\H_- \cap \H_+ = \emptyset$ and $\H_\pm \subset I_\pm(\S)$, see \cite{O'Neill1983}*{Prop. 14.53}.
\begin{definition}[Cauchy horizon]
We define $\H_+$ and $\H_-$ to be the future and past Cauchy horizon of $\S$, respectively.
\end{definition}

We are from now on going to let $\H$ denote the future \emph{or} the past Cauchy horizon of $\S$.
The following recently proven theorem is very useful for our purposes:

\begin{thm}[\cite{Larsson2015}*{Cor.\ 1.43}, \cite{Minguzzi2015}*{Thm.\ 18}] \label{thm: smoothness Cauchy horizon}
Let $M$ and $\S$ be as above. 
Assume that $\H$ is a compact Cauchy horizon of $\S$ and that $(M, g)$ satisfies the null energy condition, i.e.\ that
\[
	\Ric(L, L) \geq 0
\]
for all lightlike vectors $L \in TM$.
Then $\H$ is a smooth, totally geodesic and lightlike hypersurface.
\end{thm}

In the theorems below, we will always assume that the null energy condition is satisfied.
We may therefore from now on assume that $\H$ is a smooth, compact and lightlike hypersurface.
Since $M$ is time-oriented, there is a nowhere vanishing timelike vector field $T$ on $M$.
Since $\H$ is a lightlike hypersurface, $T|_\H$ is transversal to $\H$, so $\H$ is two-sided.
Moreover, there is a smooth one-form $\b$ such that ${\b(T)|_\H \neq 0}$ and $\b(X) = 0$ for all $X \in T\H$.
It follows that $V := \b^\sharp|_\H$ is a nowhere vanishing vector field normal to $T\H$.
Since $\H$ is lightlike, $V$ must be lightlike and tangent to $\H$.
One checks that any such lightlike vector field satisfies
\[
	\n_V V = \kappa V
\]
for a smooth function $\kappa$ on $\H$.
The function $\kappa$ is called \emph{surface gravity} of $\H$ with respect to $V$.
Note that $V$ is not canonical and the surface gravity depends on our choice of $V$.

\begin{definition} \label{def: constant surface gravity}
We say that the surface gravity of $\H$ can be normalised to a non-zero constant if there is a nowhere vanishing lightlike vector field $V$, tangent to $\H$, such that
\[
	\n_V V = \kappa V
\]
on $\H$, for some \textbf{constant} $\kappa \neq 0$. 
\end{definition}

\subsubsection{Compact Killing horizons}

Our first main result says that all compact Cauchy horizons with constant non-zero surface gravity in vacuum spacetimes are smooth Killing horizons:

\begin{thm}[Killing horizon] \label{thm: Killing Horizon main}
Let $M$ and $\S$ be as above. 
Assume that $M$ is a vacuum spacetime, i.e.\ $\Ric = 0$, and that $\H$ is a compact Cauchy horizon of $\S$, with surface gravity that can be normalised to a non-zero constant.
Then $\H$ is a smooth Killing horizon.
More precisely, there is an open subset $\O \subset M$, containing $\H$ and $D(\S)$, and a unique smooth Killing vector field $W$ on $\O$ such that
\[
	W|_\H = V,
\]
where $V$ is as in Definition \ref{def: constant surface gravity}.
Moreover, $W$ is spacelike in $D(\S)$ close to $\H$, lightlike on $\H$ and timelike on $\O \backslash (D(\S) \cup \H)$ close to $\H$. 
\end{thm}

The construction of the Killing vector field $W$ will rely on a certain null time function, see Proposition \ref{prop: null time function}.
Let us briefly explain the construction of the null time function here.
We prove in Proposition \ref{prop: null time function} that there is a unique lightlike transversal vector field $L$ along $\H$ such that\begin{align*}
	g(L, V)|_{\H}
		&= 1, \\*
	g(L, X)|_{\H}
		&= 0,
\end{align*}
for all $X \in T\H$ such that $\n_X V = 0$.
We will throughout the paper let $\d_t$ denote the geodesic vector field such that $\d_t|_\H = L$.
The coordinate obtained by flowing $\H$ along $\d_t$ (in a small open neighbourhood $\U$), will be called the \emph{null time function}
\[
	t: \U \subset M \to (-\e, \e).
\]
Our null time function was first constructed in \cite{Petersen2018} and is of central importance for understanding the geometry close to $\H$.
The Killing vector field $W$ has an explicit construction close to $\H$, in terms of the null time function:

\begin{thm} \label{thm: Killing close to horizon}
On an open neighbourhood $\U$ of the Cauchy horizon $\H$, the Killing vector field $W$ in Theorem \ref{thm: Killing Horizon main} is the unique solution to the following transport equation:
\begin{align*}
	[W, \d_t]
		&= 0, \\
	W|_{\H}
		&= V.
\end{align*}
\end{thm}

\noindent
The main work in this paper consists in showing that $W$ is a Killing vector field.
Theorems \ref{thm: Killing Horizon main} and \ref{thm: Killing close to horizon} are proven in Section \ref{sec: extension of KVFs}.
Let us compare Theorem \ref{thm: Killing Horizon main} and Theorem \ref{thm: Killing close to horizon} with the simplest example possible:

\begin{example}[The Misner spacetime] \label{ex: Misner}
Let
\[
	M = \R \times S^1, \quad g = 2dtdx - tdx^2,
\]
where $t$ and $x$ are the coordinates on $\R$ and $S^1 := \R/\Z$, respectively.
Choosing $\S := \{-1\} \times S^1$, we see that $\H := \{0\} \times S^1$ is the future Cauchy horizon and $D(\S) = (-\infty, 0) \times S^1$.
For an illustration of the light cones and different regions, see Figure \ref{fig: Regions_CauchyHorizon}.
With $V := \d_x|_\H$, the surface gravity is given by $\kappa = \frac12$, i.e.\
\[
	\n_V V = \frac12 V.
\]
Theorem \ref{thm: Killing Horizon main} therefore applies.
Indeed, in this case we have the global Killing vector field
\[
	W = \d_x
\]
and $\U = \O = M$.
The vector field $\d_x$ is spacelike on $D(\S)$, lightlike on $\H$ and timelike on $M \backslash (D(\S) \cup \H) = (0, \infty) \times S^1$.
The coordinate $t: M \to \R$ is the null time coordinate as above, with $L = \d_t|_\H$ and $[W, \d_t] = [\d_x, \d_t] = 0$.
\end{example}

For further examples, including the Taub-NUT spacetime, and general remarks on spacetimes with compact Cauchy horizons with constant non-zero surface gravity, we refer the reader to \cite{Petersen2018}*{Sec.\ 2}.

\subsubsection{Extension of other Killing vector fields}
There might of course exist more Killing vector fields on the globally hyperbolic side of the Cauchy horizon, which extend smoothly up to the Cauchy horizon.
Our second main result says in particular that all such Killing vector fields extend beyond the Cauchy horizon:

\begin{thm}[Extension of Killing vector fields] \label{thm: Killing Extension main}
Let $M$ and $\S$ be as above. 
Assume that $M$ is a vacuum spacetime, i.e.\ $\Ric = 0$, and that $\H$ is a compact Cauchy horizon of $\S$, with surface gravity that can be normalised to a non-zero constant.
Then there is an open neighbourhood $\O$, containing $\H$ and $D(\S)$, such that if a smooth vector field $Y$ satisfies
\begin{equation} \label{eq: Killing to inf order}
	\n^m \L_Yg|_\H
		= 0,
\end{equation}
for all $m \in \N_0$, then there is a unique Killing vector field $Z$ on $\O$ such that
\[
	\n^m Z|_\H = \n^m Y|_\H,
\]
for all $m \in \N_0$.
\end{thm}

\noindent 
The notation $\n^m a|_\H = 0$ for all $m \in \N_0$ means that the tensor $a$ and all its transversal derivatives vanish at $\H$.
Theorem \ref{thm: Killing Horizon main} will in fact be proven by combining Theorem \ref{thm: Killing Extension main} and \cite{PetersenRacz2018}*{Thm.\ 2.1}.

An interesting result by Isenberg and Moncrief in \cite{IM1992} states that if at least one of the orbits of the Killing vector field $W$ \emph{does not close}, then there is a \emph{second} Killing vector field $Z$ exists on the globally hyperbolic region.
We combine this observation with Theorem \ref{thm: Killing Extension main} and obtain the following corollary:

\begin{cor} \label{cor: second Killing}
Assume that at least one of the generators (integral curves of $V$) is non-closed and that $D(\S)$ is maximal globally hyperbolic.
Then there is an open subset $\O$, containing $\H$ and $D(\S)$, and a \emph{second} Killing vector field $Z$ (different from $W$) on $\O$, leaving the null time function invariant, i.e.\
\[
	[Z, \d_t] = 0,
\]
on the open subset $\U$, where $\d_t$ is defined.
In fact, the isometry group of $\U$ must have an $S^1\times S^1$ subgroup leaving the null time function invariant.
\end{cor}

Theorem \ref{thm: Killing Extension main} and Corollary \ref{cor: second Killing} are proven in Section \ref{sec: extension of KVFs}.

\subsubsection{Unique continuation for wave equations}
The main ingredient in proving Theorem \ref{thm: Killing Horizon main} and Theorem \ref{thm: Killing Extension main} is a new unique continuation theorem for wave equations coupled to transport equations.
The precise formulation is postponed to Theorem \ref{thm: unique continuation}, since we need to introduce more structure.
Let us therefore simply present here the statement for wave equations without coupling to transport equations.

\begin{definition}
Let $F \to M$ be a real or complex vector bundle.
A \emph{wave operator} is a linear second order differential operator acting on sections of $F$ with principal symbol given by the metric, i.e.\ it can locally be expressed as
\[
	\sum_{\a, \b = 0}^n -g^{\a\b}\on^2_{e_\a, e_\b} + l.o.t.,
\]
where $(e_0, \hdots, e_n)$ is a local frame and $\on$ is a connection on $F$.
\end{definition}

\noindent
Let from now on $P$ be a wave operator acting on sections of a real or complex vector bundle $F \to M$.
We will also assume that the dominant energy condition is satisfied:
\begin{definition}
A spacetime $(M,g)$ is said to satisfy the \emph{dominant energy condition} if the stress energy tensor $T := \Ric - \frac12 \mathrm S g$ satisfies the following: For any future pointing causal vector $X$, the vector $-T(X, \cdot)^\sharp$ is future pointing causal (or zero).
\end{definition}

\noindent Note that 
\[
	\mathrm{vacuum} \Rightarrow \mathrm{dominant \ energy \ condition} \Rightarrow \mathrm{null \ energy \ condition}.
\] 
Our main unique continuation theorem for wave equations coupled to transport equations is Theorem \ref{thm: unique continuation}, which has the following important special case:

\begin{thm}[Wave equations] \label{thm: Wave main}
Let $M$ and $\S$ be as above. 
Assume that $(M,g)$ satisfies the dominant energy condition and that $\H$ is a compact Cauchy horizon of $\S$, with surface gravity that can be normalised to a non-zero constant.
Then there is an open neighbourhood $\O$, containing $\H$ and $D(\S)$, such that if $u \in C^\infty(\O, F)$ satisfies
\begin{align*}
	P u 
		&= 0 \text{ on } \O, \\*
	\on^m u|_\H
		&= 0
\end{align*}
for all $m \in \N_0$, then 
\[
	u|_\O = 0.
\]
\end{thm}

\begin{example} \label{ex: Misner wave}
The simplest example to which Theorem \ref{thm: Wave main} applies is the Misner spacetime, Example \ref{ex: Misner}, where
\begin{align*}
	\Box 
		&= - \d_t(t\d_t + 2 \d_x) \\
		&= -\frac1t \left(\d_{\grad(t)}\right)^2 + \frac1t \left(\d_x\right)^2,
\end{align*}
where we used that $\grad(t) = t\d_t + \d_x$.
\end{example}

\begin{remark}
Theorem \ref{thm: Wave main} says, in particular, that one can predict solutions to linear wave equations also \emph{beyond} any compact Cauchy horizon with constant non-zero surface gravity in a spacetime satisfying the dominant energy condition.
\end{remark}

Theorem \ref{thm: Wave main} relies heavily on our assumption that $\kappa \neq 0$.
In fact, in case $\kappa = 0$, then linear waves are not predictable beyond the Cauchy horizon in general:

\begin{example}[Unique continuation fails for vanishing surface gravity] \label{ex: van surf grav}
Consider the spacetimes
\[
	M = \R \times S^1, \quad g = 2dtdx + (- t)^mdx^2
\]
for $m \in \N$.
By Example \ref{ex: Misner}, we know that the assumptions of Theorem \ref{thm: Wave main} are satisfied if $m = 1$.
A simple calculation with $V := \d_x|_\H$ shows that for $m \geq 2$ we have
\[
	\n_V V = 0,
\]
i.e.\ the surface gravity vanishes.
We now show that the conclusion in Theorem \ref{thm: Wave main} actually fails for $m \geq 2$. 
The d'Alembert operator is given by 
\[
	\Box = \d_t((-t)^m \d_t - 2 \d_x).
\]
Again it is easy to see that $\H = \{0\} \times S^1$ is the future Cauchy horizon of $\S := \{-1\} \times S^1$.
Consider the smooth function 
\[
	u(t, x) := \begin{cases} e^{-\frac1t} & t > 0, \\ 0 & t \leq 0.\end{cases}
\]
By construction, $u(t, \cdot) = 0$ for any $t \leq 0$ and $u(t, \cdot) \neq 0$ for any $t > 0$.
Note that
\[
	\Box u + \left((m-2)(-t)^{m-1} - (-t)^{m-2}\right) \d_t u = 0.
\]
If $m \geq 2$, this is a wave equation with smooth coefficients.
We conclude that unique continuation is false in general for compact Cauchy horizons of vanishing surface gravity.
\end{example}

\begin{remark}
It is interesting to note that the spacetimes $(M,g)$ in the previous example are \emph{flat} if and only if $m = 1$, which happens if and only if the surface gravity is non-zero.
As already mentioned, all known examples of compact Cauchy horizons in \emph{vacuum} spacetimes have constant non-zero surface gravity and fulfil the assumptions of Theorem \ref{thm: Killing Horizon main}, Theorem \ref{thm: Killing Extension main} and Theorem \ref{thm: Wave main}.
\end{remark}

A natural question to ask in relation to Theorem \ref{thm: Wave main} is whether it suffices to assume for example that $u|_\H = 0$ in order to conclude that $u|_\O = 0$? 
(This is true for standard characteristic Cauchy problems, c.f.\ \cite{BW2015}*{Thm.\ 23}.)
It turns out that this is false in general, see Example \ref{ex: k derivatives not enough} below.
In fact, one can in principle compute from the \emph{first order} part of the wave operator how many derivatives have to vanish at the horizon in order to conclude that the solution vanishes on an open neighbourhood. 
Let us describe this here.
Assume that $a$ is a symmetric or Hermitian positive definite scalar product on $F$ and fix a compatible connection $\on$, i.e.\
\[
	\on a = 0.
\]
We extend this connection to elements $S \otimes u$, where $S$ is a tensor field and $u$ is a section of $F$ by the product rule
\[
	\on(S \otimes u) := (\n S) \otimes u + S \otimes \on u,
\]
where $\n$ is the Levi-Civita connection with respect to $g$.
This allows us to define higher order derivatives $\on^m$, for instance
\[
	\on^2_{X, Y} u := \on_X \on_Y u - \on_{\n_X Y} u,
\]
for any vector fields $X, Y$ on $M$.
Given any wave operator $P$, we may express it as
\[
	P u = \bBox u+ B(\on u) + A u,
\]
where $B$ and $A$ are smooth homomorphism fields and $\bBox := - \tr_g(\on^2)$.
We have the following corollary of Theorem \ref{thm: Wave main}:
\begin{cor} \label{cor: finite order vanishing}
Let $l$ be the smallest integer such that 
\[
	l \geq \frac{1}{2\kappa}\frac{\Re\Big(a|_\H\big(B(g(V, \cdot) \otimes w), w \big) \Big)}{a(w, w)}
\]
for all $w \in F|_\H$ (by compactness of $\H$, such an $l$ always exists).
If $u \in C^\infty(\O, F)$ satisfies
\begin{align*}
	P u
		&= 0, \\*
	\on^m u|_\H 
		&= 0,
\end{align*}
for all $m \leq l$, then $u|_\O = 0$.
\end{cor}

\noindent 
Corollary \ref{cor: finite order vanishing} is proven in Section \ref{subsec: proof unique cont}.
It says in particular: For each wave operator there is a \emph{finite} order $l$, to which it suffices to assume that the solution vanishes, in order to conclude that it vanishes on an open neighbourhood containing $\H$ and $D(\S)$.
The statement is sharp in the sense that the order $l$, to which one has to assume that the solution vanishes, really depends crucially on the first order coefficient $B$ (the first order terms) of the wave operator:
\begin{example} \label{ex: k derivatives not enough}
Consider the Misner spacetime, Example \ref{ex: Misner}, with d'Alembert operator $\Box  = -\d_t(t\d_t + 2\d_x)$.
Note that for functions, the natural choices are simply
\[
	a(f_1, f_2) := f_1 \overline{f_2}, \quad \on := \d.
\]
In particular $\bBox = \Box$ on functions, which is not necessarily true for $\Box$ on tensors.
For each integer $k \geq 0$, note that
\begin{align*}
	\left(\bBox + (k+1)\d_t\right)t^{k+1}
		&= 0, \\
	\on^m t^{k+1}|_{\H}
		&= 0
\end{align*}
for all $m \leq k$.
In other words, for each integer $k \geq 0$ there is a smooth solution $t^{k+1}$ to a homogeneous wave equation, which is non-trivial for $t \neq 0$ and vanishes up to order $k$ at the Cauchy horizon.
\end{example}

\subsubsection{Local rigidity of stationary black holes}
Let us now explain why Hawking's local rigidity theorem without analyticity, proven by Alexakis-Ionescu-Klainerman in \cite{AIK2010}*{Thm. 1.1}, follows directly from our Theorem \ref{thm: Killing Horizon main}.
In particular, we will explain Figure \ref{fig: EH to CH} in more detail.
We begin by introducing the necessary notions for the definition of stationary black hole spacetimes.

\begin{definition}[Asymptotically flat  hypersurface]
The spacetime $(M,g)$ is said to possess an \emph{asymptotically flat hypersurface} if $M$ contains a spacelike hypersurface $S_\ext$ with a diffeomorphism 
\[
	\varphi :S_\ext \to \R^n \backslash \overline{B(R)},
\]
where $B(R)$ is the open ball of radius $R > 0$, such that the induced first and second fundamental forms $(\gamma, K)$ on $S_{\ext}$ satisfy
\begin{align*}
	(\varphi_* \gamma)_{ij} - \delta_{ij} &\in \mathcal O_k(r^{-\a}), \\
		(\varphi_*K)_{ij} &\in \mathcal O_{k-1}(r^{-1-\a}),
\end{align*}
for some $\a > 0$ and some integer $k > 1$, where $f \in \mathcal O_k(r^{-\a})$ if $\n_{i_1} \hdots \n_{i_l}f \in \mathcal O(r^{-\a-l})$ for all $l \leq k$, where here $\n$ is the induced Riemannian Levi-connection.
\end{definition}

The precise rate of decay is not important for the results here.

\begin{definition} \label{def: stationary spacetime}
We call a spacetime $M$ containing an asymptotically flat hypersurface $S_\ext$ a \emph{stationary asymptotically flat spacetime} if there exists a complete Killing vector field $K$ on $M$ which is timelike along $S_\ext$.
Let $\phi_t: M \to M$ denote the flow of $K$.
We define the \emph{exterior region} as
\[
	 M_\ext := \cup_t \phi_t(S_\ext)
\]
and the \emph{domain of outer communication} as
\[
	\langle\langle M_\ext \rangle\rangle := I^+(M_\ext) \cap I^-(M_\ext).
\]
The \emph{black hole region} is defined as
\[
	\mathcal B := M \backslash I^-(M_\ext)
\]
and the \emph{black hole event horizon} as $\H^+_{bh} := \d \mathcal B$.
Similarly, the \emph{white hole region} is defined as
\[
	\mathcal W := M \backslash I^+(M_\ext)
\]
and the \emph{white hole event horizon} as $\H^-_{bh} := \d \mathcal W$.
\end{definition}
Let us for simplicity of presentation assume that $\mathcal W = \emptyset$, i.e.\ that $M = I^+(M_\ext)$.
In particular, the past event horizon is empty and there is no bifurcation surface.
Some regularity assumption is in order.
We have chosen to follow \cite{ChruscielCosta2008}*{Def.\ 1.1} and restrict, for simplicity, to one asymptotically flat end.
\begin{assumption} \label{ass: black hole}
Let $(M,g)$ be a stationary asymptotically flat vacuum spacetime, i.e.\ $\Ric = 0$, with 
\[
	M = I^+(M_\ext),
\]
where $M_\ext$ is the exterior region as in Definition \ref{def: stationary spacetime}.
Assume that there is a closed spacelike hypersurface $S \supset S_\ext$ in $M$, with boundary $\d S$, such that $S \backslash S_\ext$ is compact and such that $\d \mathcal S$ is a compact cross-section in $\H_{bh}^+$, i.e.\ any generator (lightlike integral curve) of $\H_{bh}^+$ intersects $\d \mathcal S$ precisely once.
Assume also that $\langle \langle M_\ext \rangle \rangle$ is a globally hyperbolic spacetime and that $S$ is achronal in $\langle \langle M_\ext \rangle \rangle$.
\end{assumption}

See \cite{ChruscielCosta2008}*{Figure 1.1} for a nice picture illustrating this assumption. 
Since $\H_{bh}^- = \emptyset$, let us write $\H_{bh} := \H_{bh}^+$.
The following theorem is a special case of \cite{ChruscielCosta2008}*{Thm.\ 4.11}, which is based on \cite{CDGH2001}.
\begin{thm}[\cite{ChruscielCosta2008}*{Thm. 4.11}] \label{thm: smoothness bh event horizons}
If Assumption \ref{ass: black hole} holds, then $\H_{bh}$ is a smooth null hypersurface in $M$.
\end{thm}

Note that this theorem is analogous to Theorem \ref{thm: smoothness Cauchy horizon} for compact Cauchy horizons.
Since we always work with Assumption \ref{ass: black hole}, we may from now on assume that $\H_{bh}$ is smooth.
Similarly as before, there is a nowhere vanishing lightlike vector field $V$, tangent to $\H_{bh}$, such that
\[
	\n_V V = \kappa V.
\]
\begin{definition} \label{def: constant surface gravity event}
We say that the surface gravity of $\H_{bh}$ can be normalised to a non-zero constant if there is a nowhere vanishing lightlike vector field $V$, tangent to $\H_{bh}$, such that
\begin{align*}
	\n_V V 
		&= \kappa V, \\
	[V, K]
		&= 0,
\end{align*}
on $\H_{bh}$, for some \textbf{constant} $\kappa \neq 0$. 
Here $K$ is the Killing vector field from Definition \ref{def: stationary spacetime}.
\end{definition}

Note that this assumption is immediately satisfied if one assumes the existence of a bifurcation surface, see \cite{IK2009}*{p.\ 38}.
We will prove the following version of Hawking's local rigidity theorem for smooth stationary black hole spacetimes:

\begin{thm}[Killing event horizon] \label{thm: Killing Extension bh}
Let $(M, g)$ be a stationary asymptotically flat vacuum spacetime. 
In addition to Assumption \ref{ass: black hole}, assume that the surface gravity of $\H_{bh}$ can be normalised to a non-zero constant.
Then $\H_{bh}$ is a smooth Killing horizon.
More precisely, there exists a Killing vector field $W$, defined on an open neighbourhood $\U$ of $\H_{bh}$, such that
\[
	W|_{\H_{bh}} = V,
\]
where $V$ is as in Definition \ref{def: constant surface gravity event}.
Moreover, the subset $\U \cap \left( \H_{bh} \cup \langle\langle M_\ext \rangle\rangle \right)$ is invariant under the flow of the stationary Killing vector field $K$.
\end{thm}

\begin{remark}
Note that we make no further assumptions neither on the spacetime dimension nor on the topology of the event horizon.
\end{remark}

Essentially the statement of Theorem \ref{thm: Killing Extension bh} is due to Alexakis-Ionescu-Klainerman, by applying \cite{AIK2010}*{Thm. 1.1} to stationary black holes, in spacetime dimension $4$ with spherical cross-section topology.
See also the refined result by Ionescu-Klainerman \cite{IonescuKlainerman2013}*{Thm. 4.1}, for general topology of the cross-section.
It seems reasonable that their proof also goes through in higher dimensions.
The proof of Alexakis-Ionescu-Klainerman relies on the existence of a bifurcation surface, i.e.\ that the future and past event horizons intersect transversally in a smooth surface.
Under this assumption, the authors show that they may normalise the surface gravity to a non-zero constant, c.f.\ also \cite{RaczWald1995} for the (partly) converse statement.

We want to emphasise that our method to prove Theorem \ref{thm: Killing Extension bh} does \emph{not} use the existence of a bifurcation surface.
Our argument is based on the fact that a neighbourhood of the event horizon can be viewed as a covering space over a neighbourhood of a compact Cauchy horizon (this observation was first used in \cite{FRW1999}).
This is what Figure \ref{fig: EH to CH} illustrates.
The event horizon covers the Cauchy horizon and the null time function is covered by a certain \enquote{ingoing/outgoing null coordinate}.
Theorem \ref{thm: Killing Extension bh} then follows as a corollary from Theorem \ref{thm: Killing Horizon main}, by just lifting the Killing vector field using this covering map.
We thus obtain an alternative proof of Hawking's local rigidity theorem for smooth stationary asymptotically flat vacuum black holes, with event horizons of constant surface gravity, relying on a unique continuation theorem which is independent of that by Alexakis-Ionescu-Klainerman.

\subsubsection{Relation to previous literature}Unique continuation for wave equations through lightlike or other types of degenerate hypersurfaces is a classical topic of interest.
The simplest case is uniqueness on (a subset of) the domain of dependence of the lightlike hypersurface. 
The most general such result is due to B\"{a}r and Wafo in \cite{BW2015}*{Thm.\ 23}.
Their proof is based on elementary energy estimates, as opposed to Carleman estimates.

Their result does, however, not apply in many important cases, since the domain of dependence lightlike or otherwise degenerate hypersurfaces of interest are often nothing but the hypersurface itself (as is the case for compact Cauchy horizons).
In the unique continuation theorem by Ionescu-Klainerman in \cite{IK2009_2}, on which their analysis of stationary black holes is based, the unique continuation is proven from a \emph{bifurcate} lightlike hypersurface.
It is interesting to consider in what sense H\"{o}rmander's pseudo-convexity fails in their case (making H\"{o}rmander's unique continuation theorem \cite{Hormander1985}*{Thm.\ 28.3.4} inapplicable).
Ionescu-Klainerman consider a sequence of pseudo-convex hypersurfaces approaching the bifurcate hypersurface and show that the Carleman estimates do not degenerate in the limit, proving the unique continuation.
An important feature of their result, is that the statement is \emph{local} near the bifurcation surface.

Our Theorem \ref{thm: Wave main} is the first unique continuation result for wave equations through a \emph{smooth} lightlike hypersurface (apart from our \cite{Petersen2018}*{Cor.\ 1.8}, which is a one-sided version of this).
It is very different in nature from the theorem of Ionescu-Klainerman, since we need \emph{global} assumptions along the compact lightlike hypersurface, because we do not work with any bifurcation.
We can view our Carleman estimate as a limit of Carleman estimates for pseudo-convex hypersurfaces, see Remark \ref{rmk: pseudo-convex} for a discussion on this.
Our non-degeneracy assumption of non-vanishing surface gravity $\kappa$, ensures that the pseudo-convexity is violated \emph{only to first order} at the compact Cauchy horizon, which is the reason why the Carleman estimate is true.
Indeed, Example \ref{ex: van surf grav} shows that unique continuation, and therefore the Carleman estimates, is false when $\kappa = 0$.

Rather than with the result of Ionescu-Klainerman, our techniques actually have more in common with the proof of unique continuation from infinity for linear waves by Alexakis-Schlue-Shao \cite{ASS2016} and the unique continuation theorems from conformal infinity in asymptotically anti-de Sitter spacetimes by Holzegel-Shao \cite{HS2017} and in asymptotically de Sitter spacetimes by Vasy \cite{Vasy2010}.

In both \cite{ASS2016} and \cite{HS2017}, it is not possible to localize the unique continuation statement, as was possible in the work of Ionescu-Klainerman.
The authors therefore have to be very careful in choosing the sequence of pseudo-convex hypersurfaces in a way that ensures that the pseudo-convexity does not degenerate too fast, in order to obtain Carleman estimates in the limit.
This is very similar to the situation in this paper.

Moreover, in \cite{ASS2016}, \cite{HS2017} and \cite{Vasy2010} one needs to assume high order vanishing in order to get the unique continuation.
This is related to the fact that the Carleman estimates are \emph{singular}, which in turn is related to the fact that the pseudo-convexity degenerates drastically at null infinity and conformal infinity, respectively.
Also this is analogous to the situation in the present paper, where we are forced to use singular Carleman estimates and therefore need to assume higher vanishing of the functions at the horizon in order to conclude unique continuation.
This is not necessary in the problem discussed by Ionescu-Klainerman, where it suffices to assume that the function vanishes and one does not need to assume anything about the derivatives.

Finally, wave equations close to a Cauchy horizon are to a certain extent reminiscent of wave equations of Fuchsian type.
Fuchsian wave equations show up naturally in spacetimes close to the initial big bang singularity, under certain conditions.
Most of the results are done in the analytic setting, see \cite{AnderssonRendall2001} and references therein. 
Some more recent results dropped the assumption of analyticity, see \cite{ABIL2013}, \cite{ABIL2013_2}, \cite{BH2012}, \cite{BH2014}, \cite{Rendall2000}, \cite{BL2010} and \cite{Stahl2002} and references therein. 
See also the recent work by Rodnianski-Speck, where they prove stability towards the singular direction of a Kasner-like singularity \cite{RodnianskiSpeck2018}, \cite{RodnianskiSpeck2018_2}.
The main difference to our work is that the causal structure in these spacetimes is \enquote{silent} close to the singularity, which is very different from the setting in this paper.
Moreover, we neither assume analyticity nor close to symmetry of the spacetime.

\subsection{Strategy of the proofs} \label{subsec: Strategy}

Let us start by recalling how we proved in \cite{PetersenRacz2018} that $W$ (defined here in Theorem \ref{thm: Killing close to horizon}) is a Killing vector field on the globally hyperbolic side of the Cauchy horizon. 
This will clarify the difficulty in showing that $W$ is a Killing vector field beyond the Cauchy horizon.
The first step is to show that $W$ solves the Killing equation up to any order at the Cauchy horizon.
This computation is the main novelty in \cite{PetersenRacz2018}, generalising classical work by Moncrief and Isenberg \cite{MoncriefIsenberg1983}.
A Killing vector field $\widehat W$ (which a posteriori is shown to coincide with $W$) is then constructed on the globally hyperbolic region $D(\S)$ by solving the linear wave equation
\begin{align}
	\Box \widehat W 
		&= 0,	 \label{eq: wave prop} \\
	\n^m \widehat W|_\H
		&= \n^m W|_\H \label{eq: wave prop ID}
\end{align}	
for any $m \in \N_0$.
The solvability of the system (\ref{eq: wave prop}-\ref{eq: wave prop ID}) on $D(\S) \cup \H$ is guaranteed by \cite{Petersen2018}*{Thm.\ 1.6}, in which the author proved that linear wave equations can be solved on $D(\S)$ given initial data on $\H$. 
Using $\Ric = 0$, a direct consequence of (\ref{eq: wave prop}-\ref{eq: wave prop ID}) is that the Lie derivative $\L_{\widehat W} g$ solves the homogeneous wave equation
\begin{align}
	\Box \L_{\widehat W}g - 2 \mathrm{Riem}(\L_{\widehat W}g) 
		&= 0, \label{eq: wave Lie} \\
	\n^m \L_{\widehat W}g|_\H
		&= 0 \nonumber
\end{align}
for any $m \in \N_0$, where $\mathrm{Riem}(\L_{\widehat W}g)$ is a certain linear combination of $\L_{\widehat W}g$ and the curvature tensor.
The uniqueness part of \cite{Petersen2018}*{Thm.\ 1.6} then proves that $\L_{\widehat W}g = 0$ on $D(\S)$, which means that $\widehat W$ is a Killing vector field on $D(\S) \cup \H$.
One finally checks that in fact $[\widehat W, \d_t] = 0$ and therefore $\widehat W = W$, which implies that $W$ is a Killing vector field.

Now, \cite{Petersen2018}*{Thm.\ 1.6} strongly relies on the fact that $D(\S)$ is globally hyperbolic.
It is not at all clear how to solve the wave equation \eqref{eq: wave prop} \emph{beyond} $\H$, since the spacetime contains closed causal curves beyond $\H$.
However, the main result of this paper implies that unique continuation for linear wave equations still holds beyond $\H$, \emph{though existence may not hold}.
If we knew that $\L_Wg$ satisfied a system of linear homogeneous wave equations like \eqref{eq: wave Lie} beyond $\H$, we would therefore conclude that $\L_Wg = 0$, which is what we want to prove.
Since \eqref{eq: wave Lie} relied on \eqref{eq: wave prop}, we cannot use \eqref{eq: wave Lie}.
Remarkably, however, a \emph{closed} such system of linear homogeneous wave equations, coupled to linear transport equations, was recently discovered by Ionescu-Klainerman in \cite{IonescuKlainerman2013}*{Prop. 4.10}.
This means that our unique continuation theorem is enough to prove that $W$ is a Killing vector field, we do \emph{not} need to prove any existence theorem beyond the Cauchy horizon.

We start out in Subsection \ref{subsec: null time function} by recalling the construction of our ``null time function". 
As already mentioned, this is a certain foliation of an open neighbourhood of the Cauchy horizon, which we essentially constructed in our earlier work \cite{Petersen2018}.
The most general form of our unique continuation statement, Theorem \ref{thm: unique continuation}, is formulated in terms of the null time function in Subsection \ref{subsec: unique cont}.

The rest of Section \ref{sec: unique cont} is devoted to the proof of Theorem \ref{thm: unique continuation} and its special case Theorem \ref{thm: Wave main}. 
The main ingredient in the proof is our singular Carleman estimate, Theorem \ref{thm: Carleman estimate}.
Let us briefly introduce the estimate here, the details are in Subsection \ref{subsec: stating Carleman}.
Denoting the null time function $t$, with $\H = t^{-1}(0)$, we consider the conjugate wave operator
\[
	\bBox_\a u := t^{-\a}\bBox(t^\a u),
\]
where $\a$ is any large enough integer and where
\[
	\bBox := - \tr_g\left( \on^2 \right)
\]
is the connection-d'Alembert wave operator.
The goal is to prove the Carleman estimate 
\begin{equation} \label{eq: Carleman prototype}
	\norm{\bBox_\a u}_{L^2} \geq C \norm{u}_{H^1_\a},
\end{equation}
or equivalently
\[
	\norm{t^{-\a} \bBox u}_{L^2} \geq C \norm{t^{-\a} u}_{H^1_\a},
\]
where $\norm{\cdot}_{H^1_\a}$ is a certain Sobolev norm with a weight dependent on $\a$. 
This Carleman estimate is the main analytic novelty in this paper.
The unique continuation theorem easily follows from this.

To illustrate the Carleman estimate, let us present it in the simple special case of the Misner spacetime (Examples \ref{ex: Misner} and \ref{ex: Misner wave}).
Note that for complex functions, we simply have $\bBox = \Box$.
For each $\a \in \N_0$, a straightforward computation shows that
\begin{align*}
	 \norm{\Box_\a u}_{L^2}^2 
	 	&= \norm{\left(\Box - \frac{\a^2}t\right) u}_{L^2}^2 + 4 \a^2 \norm{\frac{\d_{\grad(t)}u}t}^2_{L^2} \\
	 	&\quad + 2\a \left(\norm{\frac{\d_{\grad(t)}u}t}_{L^2}^2 + \norm{\frac{\d_xu}t}_{L^2}^2 + \a^2 \norm{\frac{u}t}_{L^2}^2\right)
\end{align*}
for all $u \in C^\infty_c(M)$, which vanish to infinite order at $\H = t^{-1}(0)$.
From this, one easily deduces a Carleman estimate of the form \eqref{eq: Carleman prototype}.
The surprising observation here is that the estimate is actually an equality, with only positive terms on the right hand side! 
Therefore, this novel Carleman estimate, with a singular weight function $t^\a$, turns out to come very naturally with the Misner spacetime.
The Carleman estimate for general compact Cauchy horizons with non-zero surface gravity will to a certain extent be similar to this, but will of course not be an equality in general.

The first step in proving the Carleman estimate in the general case is to split the conjugate operator $\bBox_\a$ into formally self-adjoint and anti-self-adjoint parts $\bBox_\a^s$ and $\bBox_\a^a$.
Up to lower order terms, we prove that
\begin{align*}
	\bBox_\a^s 
		&\approx \bBox - \frac{\a^2}{t},\\
	\bBox_\a^a
		&\approx - \frac{2\a}t \on_{\grad(t)}.
\end{align*}
By the equality
\begin{align}
	\norm{\bBox_\a u}^2_{L^2} 
		&= \norm{\bBox_\a^s u}^2_{L^2} + \norm{\bBox_\a^a u}^2_{L^2} + \ldr{\bBox_\a^a u, \bBox_\a^s u}_{L^2} + \ldr{\bBox_\a^s u, \bBox_\a^a u}_{L^2} \nonumber \\*
		&= \norm{\bBox_\a^s u}^2_{L^2} + \norm{\bBox_\a^a u}^2_{L^2} + \ldr{(\bBox_\a^s \bBox_\a^a - \bBox_\a^a\bBox_\a^s) u, u}_{L^2}, \label{eq: Box alpha}
\end{align}
it is clear that the crucial term to estimate is
\begin{equation}
	\ldr{[\bBox_\a^s, \bBox_\a^a]u, u}_{L^2} \approx \ldr{\left[ \bBox - \frac{\a^2}{t}, - \frac{2\a}t \on_{\grad(t)}\right]u, u}_{L^2}. \label{eq: commutator}
\end{equation}
One main difficulty is to prove a lower bound for this term close to $\H$, i.e.\ for $t \in (-\e, \e)$, where $\e > 0$ is small.
Surprisingly, it turns out that this can be done without any further assumptions (than the ones made above) concerning the geometry of the Cauchy horizon or the dimension of the spacetime.

The proof is based on determining the asymptotic behaviour of the spacetime metric as $t \to 0$, i.e.\ close to the horizon.
We perform a fine analysis of the asymptotic behaviour of each component of the metric with respect to a suitable frame in Subsection \ref{subsec: properties of t}. 
As one might expect from commuting $\Box$ with $\n_{\grad(t)}$, the Hessian of the null time function also plays an important role.
We prove that the Hessian of the null time function can be computed up to quadratic errors in $t$ as $t \to 0$. 
We then use this to prove the Carleman estimate in Subsection \ref{subsec: proof Carleman}.
Our estimate can easily be coupled to a corresponding one for transport equations.
Using the coupled Carleman estimates with $\a \to \infty$, we prove the unique continuation statement, Theorem \ref{thm: unique continuation}, in Subsection \ref{subsec: proof unique cont}.

There are two important differences to standard Carleman estimates, like Hörmander's classical theorem \cite{Hormander1985}*{Thm.\ 28.3.4}.
The weight function $t^{-\a}$ is \emph{singular} at $t = 0$ and is defined \emph{along the entire hypersurface} $\H$ and not just in a small open subset of $\H$ (which is usually the case when unique continuation is studied).
Since $\H$ is lightlike (characteristic), the argument would fail if the weight function did not satisfy both these properties. 
Indeed, Hörmander's classical unique continuation theorem does not apply to lightlike hypersurfaces.
This makes our argument \enquote{non-local} in this certain sense, a local argument on for example a coordinate patch would not suffice!

Let us now briefly explain how we apply our results to stationary black hole spacetimes.
Any stationary vacuum black hole spacetime with an event horizon with constant non-zero surface gravity can be viewed as a covering space over a vacuum spacetime with a compact Cauchy horizon, this is illustrated in Figure \ref{fig: EH to CH}.
The Cauchy horizon is lifted to the future \emph{or} the past event horizon in the covering black hole spacetime.
Let us explain explicitly how this is done in the simplest example of a black hole spacetime.
The union of the domain of outer communication and the black hole region in the Schwarzschild spacetime can be written in ingoing null coordinates as follows:
\begin{align*}
	\widetilde M 
		&:= \R_{> 0} \times \R \times S^2\\
	\widetilde g 
		&:= 2dv dt - \left(1 - \frac{2m}{t}\right)dv^2 + t^2 g_{S^2},
\end{align*}
were $t$ is the usual \enquote{radial} coordinate on $\R_{> 0}$ (most commonly denoted $r$).
The reason we denote it by $t$ here is that it is exactly the \emph{null time function} to the event horizon $\H_{bh} := \{t = 2m\}$.
Note that $W := \d_v$ is a Killing vector field, which at the event horizon is tangent and lightlike.
Moreover, the flow of $W$ induces a group of isometries of $\widetilde M$, which acts free and proper. 
We may therefore pass to the quotient 
\begin{align*}
	M 
		&:= \R_{> 0} \times S^1 \times S^2\\
	g 
		&:= 2dv dt - \left(1 - \frac{2m}{t}\right)dv^2 + t^2 g_{S^2}.
\end{align*}
The event horizon in $\widetilde M$ becomes in the quotient $M$ a compact future Cauchy horizon 
\[
	\H := \{2m\} \times S^1 \times S^2
\]
(with the convention that $\d_t$ is future directed).
This is schematically illustrated in Figure \ref{fig: EH to CH} (though the Schwarzschild singularity at $t = 0$ is not present in that figure).
We conclude that $\widetilde M$ \emph{covers} $M$ and the (non-compact) event horizon covers the compact Cauchy horizon.
The Killing vector field $W := \d_v$ on $M$ is lifted to the Killing vector field $\widetilde W := \d_v$ on $\widetilde M$.
The proof of Theorem \ref{thm: Killing Extension bh} is based on a generalisation of this construction, applying Theorem \ref{thm: Killing Horizon main} to prove the existence of a Killing vector field on the quotient and then lifting it to the covering black hole spacetime.
Let us emphasise that not every vacuum spacetime with a compact Cauchy horizon can be covered by a black hole spacetime. 
One such example is the classical Taub-NUT spacetime, see e.g.\ \cite{Petersen2018}*{Sec.\ 2}.

To sum up, Theorem \ref{thm: Wave main} and Corollary \ref{cor: finite order vanishing} are proven in Section \ref{sec: unique cont}, Theorems \ref{thm: Killing Horizon main}, \ref{thm: Killing close to horizon} and \ref{thm: Killing Extension main} and Corollary \ref{cor: second Killing} are proven in Subsection \ref{subsec: compact Cauchy} and Theorem \ref{thm: Killing Extension bh} is proven in Subsection \ref{subsec: event horizon}.

\section{The unique continuation theorem} \label{sec: unique cont}

The purpose of this section is to present and prove our unique continuation theorem for linear wave equations coupled to linear transport equations, Theorem \ref{thm: unique continuation}.
Since we want to apply the theory to both Cauchy horizons and event horizons, it will be convenient to prove the theorem for a general compact lightlike hypersurface $N$ of constant non-zero surface gravity.
We do not assume that $N$ is a Cauchy horizon.

\begin{assumption} \label{ass: N compact}
Assume that $N \subset M$ is a non-empty, smooth, compact (without boundary), lightlike hypersurface with surface gravity that can be normalised to a non-zero constant.
Assume moreover that
\[
	\Ric(V, X) = 0
\]
for all $X \in TN$, where $\Ric$ is the Ricci curvature of $M$.
\end{assumption}
\noindent 
Throughout this section, let $N$ satisfy Assumption \ref{ass: N compact}.
We will later apply the results of this section to compact Cauchy horizons with $N = \H$.
Since event horizons of black holes are non-compact, we will first need to take a certain quotient of the event horizon, using the stationary Killing field, and then apply the results with $N = \H_{bh}/{\sim}$.

\subsection{The null time function} \label{subsec: null time function}

In order to formulate the unique continuation theorem, we need to construct a certain foliation of an open neighbourhood of $N$, as briefly described in the introduction.
We follow the strategy we developed in \cite{Petersen2018}*{Prop.\ 3.1}, with slight modifications.
The main difference is that in \cite{Petersen2018}*{Prop.\ 3.1}, the neighbourhood was one-sided, whereas here it will be two-sided.

Recall that $V$ is a nowhere vanishing lightlike vector field tangent to $N$, such that $\n_V V = \kappa V$ for some non-zero constant $\kappa$.
By substituting $V$ by $\frac 1 {2\kappa} V$, we may assume from now on that $\kappa = \frac12$.
We may also without loss of generality choose the time orientation so that $V$ is past directed.

\begin{prop}[The null time function] \label{prop: null time function}
There is an open subset $\U \subset M$ containing $N$ and a unique nowhere vanishing future pointing lightlike vector field $\d_t$ on $\U$, such that
\begin{align*}
	\n_{\d_t}\d_t 
		&= 0, \\
	g(\d_t, V)|_N
		&= 1, \quad
	g(\d_t, X)|_N
		= 0
\end{align*}
for all $X \in TN$ with $\n_XV = 0$, and such that each integral curve of $\d_t$ intersects $N$ precisely once.
Moreover, there is a unique smooth function
\[
	t: \U \to (-\e, \e),
\]
such that
\begin{align*}
	dt(\d_t) 
		&= 1, \\
	t^{-1}(0)
		&= N.
\end{align*}
Shrinking $\U$ and $\e$ if necessary, we get a diffeomorphism between $\U$ and $(-\e, \e) \times N$, where the first component is $t$.
\end{prop}

\begin{definition}
We call the function $t$ given by Proposition \ref{prop: null time function} the \emph{null time function} associated to $N$. 
\end{definition}

The value of $\e > 0$ will be changed a finite number of times throughout Section \ref{sec: unique cont}, without explicitly mentioning it.
Compare Proposition \ref{prop: null time function} with Example \ref{ex: Misner} and Figure \ref{fig: Regions_CauchyHorizon}, where the $t$-coordinate is exactly the null time function.

\begin{proof}
We begin by proving that the null second fundamental form of $N$ vanishes, i.e.\ that $N$ is totally geodesic.
This follows a standard argument.
Since $V$ is a nowhere vanishing vector field, the quotient vector bundle
\[
	TN/{\R V}
\]
is well-defined.
The null Weingarten map, defined by
\begin{align*}
	b: TN/{\R V} 	&\to TN/{\R V}, \\
		[X]					&\mapsto [\n_X V],
\end{align*}
is well-defined since $\n_V V = \frac12 V$.
Rescaling the integral curves of $V$ to (lightlike) geodesics, one observes that the geodesics are complete in the positive direction of $V$, i.e.\ they are past complete (since $V$ is past directed).
It then follows by \cite{Galloway2001}*{Prop. 3.2} that the expansion $\theta := \tr(b)$ satisfies $\theta \leq 0$ everywhere.
By \cite{Larsson2015}*{Lem. 1.3}, there is a Riemannian metric $\sigma$ on $N$ such that the induced volume density $d \mu_\sigma$ satisfies
\[
	\L_Vd\mu_\sigma = -\theta d \mu_\sigma.
\]
Since $\theta \leq 0$, it follows that the total volume of $N$ measured by $\sigma$ grows along the flow of $V$. 
But since $N$ is a compact hypersurface which is mapped diffeomorphically into itself, under the flow of $V$, the volume stays constant and we conclude that $\theta = 0$.
From the Raychaudhuri equation \cite{Galloway2001}*{Eq.\ (A.5)} and since $\Ric(V,V)|_N = 0$, it now follows that also the trace-free part of $b$ vanishes. 
We conclude that $b = 0$, i.e.\ that
\[
	g(\n_X V, Y) = 0
\]
for all $X, Y \in TN$.

Since $N$ is lightlike, this implies that there is a smooth one-form $\o$ on $N$ such that
\[
	\n_XV = \o(X)V
\]
for all $X \in TN$.
Since $\n_V V = \frac12 V$, we know that $V$ is nowhere in $\ker(\o)$.
We obtain the split into vector bundles
\begin{equation}
	TN = \R V \oplus \ker(\o). \label{eq: split TN}
\end{equation}
Since $N$ is a lightlike hypersurface, it follows that $\ker(\o) \subset TN \subset TM|_N$ is a Riemannian subbundle.
Therefore $\ker(\o)^\perp \subset TM|_N$ is a Lorentzian subbundle.
Recall that, by assumption, $M$ is time-oriented.
This means that there is a nowhere vanishing timelike vector field $T$ along $N$.
Projecting $T$ onto $\ker(\o)^\perp$ gives a nowhere vanishing vector field $T^\perp$, which is transversal to $TN$.
This implies that $\ker(\o)^\perp$ is a trivial Lorentzian vector bundle spanned by $V$ and $T^\perp$.
Since $V$ is past directed by assumption, there is a unique nowhere vanishing \emph{future pointing lightlike} vector field $L$ along $N$ such that $g(L, V) = 1$ and $g(L, X) = 0$ for any $X \in \ker(\o)$.

Let us now solve the geodesic equation from $N$ in the direction of $-L$ and $L$.
More precisely, define the map
\begin{align*}
	F: (-\e, \e) \times N &\to M, \\
	(s, p) &\mapsto \exp|_p(Ls).
\end{align*}
Since $N$ is compact, there is a small $\e > 0$ such that $F$ is a diffeomorphism onto its image $\U$.
We define $\d_t$ as the vector field with integral curves given by
\[
	s \to F(s, p),
\]
for any $p \in N$. By construction, we have $\n_{\d_t}\d_t = 0$, $g(\d_t|_N, V) = g(L, V) = 1$ and $\d_t|_N = L \perp \ker(\o)$, as claimed.
Considering the first component of the inverse of $F$, we get the uniquely determined time function
\[
	t: \U \to (-\e, \e).
\]
In particular, we get a diffeomorphism
\[
	\U = t^{-1}(-\e, \e) \cong (-\e, \e) \times N,
\]
where $\U$ is an open subset of $M$ containing $N$.
\end{proof}

From now on, we identify any subset of the form $(-\e, \e) \times N$ with the open subset $t^{-1}(-\e, \e) \subset M$.
Moreover, we identify $\{0\} \times N$ with $N$.
\begin{remark}\label{rmk: splits}
We now define the vector field $W$ on $(-\e, \e) \times N$ as the unique solution to 
\begin{align}
	[W, \d_t]
		&= 0, \label{eq: commutator W and d_t}\\
	W|_N 
		&= V. \label{eq: intial data W}
\end{align}
Later on in Section \ref{sec: extension of KVFs}, when we assume that $\Ric = 0$ on $(-\e, \e) \times N$, $W$ will indeed be shown to be a Killing vector field as claimed in the introduction.
However, for now, $W$ is just a natural extension of the vector field $V$ using the null time function.
The vector fields $\d_t$ and $W$ will be linearly independent on $(-\e, \e) \times N$.
We now define the vector bundle $E$ over $(-\e, \e) \times N$ as the span of any vector field $X$ such that
\begin{align*}
	[\d_t, X] 
		&= 0, \\
	X|_N 
		&\in \ker(\o).
\end{align*}
In other words, we Lie transport $\ker(\o)$ along $\d_t$.
Since $[W, \d_t] = 0 = [\d_t, \d_t]$, we get the following splits:
\begin{equation} \label{eq: tangent bundle split}
	TM|_{(-\e, \e) \times N} = \R\d_t \oplus \R W \oplus E
\end{equation}
and
\begin{equation} \label{eq: hypersurface split}
	T\left(\{\tau\} \times N \right) = \R W \oplus E
\end{equation}
for all $\tau \in (-\e, \e)$.
Of course, not all vector fields $Y$ in $E$ satisfy $[\d_t, Y] = 0$, however, note that $[\d_t, Y]$ is in $E$.
Now, since $E|_N = \ker(\o)$, where $\o$ was defined in the proof of Proposition \ref{prop: null time function}, the spacetime metric $g$ is positive definite on $E|_N = \ker(\o)$. 
By compactness of $N$, let us choose $\e$ so small that $g$ is positive definite on $E$.
\end{remark}
Whenever we write $X \in E$, we mean that $X$ is a smooth vector field on $(-\e, \e) \times N$ such that $X|_p \in E$, for every $p \in (-\e, \e) \times N$.

\subsection{Formulating the theorem} \label{subsec: unique cont}
We may now formulate our unique continuation theorem for linear wave equations, coupled to linear transport equations, in terms of the null time function of the previous subsection.
Let $F \to M$ be a real or complex vector bundle and let $a$ be a positive definite symmetric or hermitian metric on $F$.
For any subset $\U \subset M$, let
\[
	C^\infty(\mathcal \U, F)
\] 
denote the smooth sections in $F$ defined on $\U$.
Let $\on$ be a compatible connection, i.e.\
\[
	\on a = 0.
\]
We extend $\on$ to elements of the form $S \otimes u$, for a tensor field $S$ and a section $u$ in $F$, by the product rule
\[
	\on(S \otimes u) := (\n S) \otimes u + S \otimes (\on u),
\]
where $\n$ denotes the Levi-Civita connection with respect to $g$.
In particular, the second derivative is given by
\[
	\on^2_{X, Y} u := \on_X \on_Y u - \on_{\n_X Y}u.
\]
We define the linear wave operator
\[
	\bBox := - \tr_g\left(\on^2\right).
\]
In a local frame, we may express $\Box$ as
\[
	\bBox = - g^{\a\b}\left(\on_{e_\a}\on_{e_\b} - \on_{\n_{e_\a}e_\b}\right),
\]
where $\n$ is the Levi-Civita connection with respect to $g$.
Let us from now on use the notation
\[
	\on_t := \on_{\d_t}.
\]
and let us define
\begin{equation} \label{eq: abs bar nabla}
	\abs{\en u}^2 := \sum_{i,j = 2}^n g^{ij}a(\on_{e_i} u, \on_{e_j}u),
\end{equation} 
expressed in some local frame $e_2, \hdots, e_n$ of $E$.
Since $E \subset TM$ is a vector subbundle on $(-\e, \e) \times N$, the definition of $\abs{\en u}$ is independent of the choice of local frame.
By Remark \ref{rmk: splits}, the metric $g$ is positive definite on $E$, which shows that the right hand side of \eqref{eq: abs bar nabla} is indeed non-negative.
The following is our main unique continuation theorem:

\begin{thm}[Unique continuation] \label{thm: unique continuation}
Let $M$ and $N$ satisfy Assumption \ref{ass: N compact} and let $F_1, F_2 \to M$ be real or complex vector bundles, equipped with compatible positive definite metrics and connections.
There is an $\e > 0$, such that if 
\[
	u_1 \in C^\infty((-\e, \e) \times N, F_1), \quad u_2 \in C^\infty((-\e, \e) \times N, F_2)
\]
satisfy
\begin{equation} \label{eq: pointwise bound assumption}
	\abs{\bBox u_1} + \abs{\on_t u_2} 
		\leq \frac C{\abs{t}} \big(\abs{\on_W u_1} + \abs{u_1} + \abs{u_2} \big) + C \big(\abs{\on_t u_1} + \abs{\en u_1}\big)
\end{equation}
for some constant $C > 0$ and
\begin{align*}
	(\on_t)^m u_1|_N
		&= 0, \\
	(\on_t)^m u_2|_N
		&= 0
\end{align*}
for all $m \in \N$, then
\begin{align*}
	u_1 
		&= 0, \\
	u_2
		&= 0
\end{align*}
on $(-\e, \e) \times N$.
The constant $C$ in \eqref{eq: pointwise bound assumption} is allowed to depend on $u_1$ and $u_2$, whereas $\e$ is independent of $u_1$, $u_2$ and $C$.
\end{thm}

\begin{remark}
Setting $u_1 = 0$ or $u_2 = 0$ in Theorem \ref{thm: unique continuation} gives unique continuation theorems for linear wave equations and linear transport equations, respectively.
\end{remark}

Let us prove Theorem \ref{thm: Wave main}, which is a simple consequence of Theorem \ref{thm: unique continuation}.

\begin{proof}[Proof of Theorem \ref{thm: Wave main}]
Since $P$ is a wave operator, there are smooth homomorphism fields $A$ and $B$, such that
\[
	\bBox u + B(\on u) + A u = 0.
\]
By the split \eqref{eq: tangent bundle split}, we obtain the pointwise estimate
\[
	\abs{\bBox u} \leq C \left( \abs{\on_t u} + \abs{\on_W u} + \abs{\en u} + \abs{u}\right).
\]
Applying Theorem \ref{thm: unique continuation} with $u_1 := u$, $u_2 = 0$ and $N = \H$ implies that $u = 0$ on $(-\e, \e) \times \H$.
Applying \cite{Petersen2018}*{Cor.\ 1.8} now proves the theorem.
\end{proof}

\begin{remark}
Recall from Example \ref{ex: van surf grav} that the smooth function
\[
	u(t, x) := \begin{cases} e^{-\frac1t} & t > 0, \\ 0 & t \leq 0.\end{cases}
\]
satisfies the equation
\[
	\Box u - \left(1 - \frac1t\right)\d_t u = 0
\]
on the Misner spacetime, $M = \R \times S^1$ with $g = 2dtdx - tdx^2$.
Hence our assumption \eqref{eq: pointwise bound assumption} is sharp in the sense that unique continuation is false in general, only assuming pointwise bounds of the form
\[
	\abs{\bBox u} \leq \frac C{\abs t}\abs{\on_t u}.
\]
\end{remark}

Theorem \ref{thm: unique continuation} will be a consequence of the Carleman estimate formulated in the next subsection.

\subsection{The Carleman estimate} \label{subsec: stating Carleman}
Given a real or complex vector bundle $F \to M$ with positive definite metric $a$ and compatible connection $\on$, let us define the vector space
\[
	C_*^\infty((-\e, \e) \times N, F) := \{u \in C_c^\infty((-\e, \e) \times N, F) : (\on_t)^mu|_N = 0 \ \forall m \in \N_0\}.
\]
In other words, $C_*^\infty((-\e, \e) \times N, F)$ denotes the compactly supported sections such that the section and all transversal derivatives vanish at $N$.

It turns out that a certain norm on $C_*^\infty((-\e, \e) \times N, F)$ is relevant for the Carleman estimates.
For this, we first define the $L^2$-inner product as
\[
	\ldr{u, v}_{L^2} := \int_{(-\e, \e) \times N} a(u, v) d\mu_g,
\]
where $u, v \in C_c^\infty((-\e, \e) \times N, F)$ and $d\mu_g$ is the induced volume density.
This induces the $L^2$-norm
\[
	\norm{u}_{L^2} := \sqrt{\ldr{u, u}_{L^2}}.
\]
We use the notation
\[
	\norm{\en u}_{L^2} := \norm{\abs{\en u}}_{L^2},
\]
where $\abs{\en u}$ is defined in \eqref{eq: abs bar nabla}.
For any $\a \in \N$, define the norm
\[
	\norm{u}_{H^1_\a}^2 := \norm{\frac1t \on_{\grad(t)}u}_{L^2}^2 + \norm{\frac1t \on_W u}_{L^2}^2 + \norm{\en u}_{L^2}^2 + \norm{\frac{\a u}t}_{L^2}^2
\]
on $C_*^\infty((-\e, \e) \times N, F)$.
The norm is well-defined though the coefficients are singular, since any section $C_*^\infty$ decays faster than any $t^m$ as $t \to 0$. 

Theorem \ref{thm: unique continuation} will be proven in Subsection \ref{subsec: proof unique cont} using the following two Carleman estimates:

\begin{thm}[The Carleman estimate for linear wave operators]\label{thm: Carleman estimate}
Let $M$ and $N$ satisfy Assumption \ref{ass: N compact} and let $F \to M$ be a real or complex vector bundle.
There are constants $\e, \a_0, C > 0$, such that 
\[
	\norm{t^{-\a} \bBox u}_{L^2} \geq \sqrt \a C \norm{t^{-\a}u}_{H^1_\a}
\]
for all $u \in C_*^\infty((-\e, \e) \times N, F)$ and all integers $\a \geq \a_0$.
\end{thm}

Let us emphasise that the constant $C$ in Theorem \ref{thm: Carleman estimate} is independent of $\a$ and $u$.
We prove Theorem \ref{thm: Carleman estimate} in the next two subsections.
\begin{prop}[The Carleman estimate for linear transport operators] \label{prop: Carleman estimate transport}
Let $M$ and $N$ satisfy Assumption \ref{ass: N compact}.
There is an $\e > 0$ such that
\[
	\norm{t^{-\a}\on_t u}_{L^2} \geq \a \norm{t^{-\a-1}u}_{L^2},
\]
for all $u \in C_*^\infty((-\e, \e) \times N, F)$ and all integers $\a \geq 0$.
\end{prop}
The proof of Proposition \ref{prop: Carleman estimate transport} is rather simple:
\begin{proof}[Proof of Proposition \ref{prop: Carleman estimate transport}]
Note that the formal adjoint of $\on_t$ is given by
\[
	(\on_t)^* = - \on_t - \div(\d_t).
\]
Using that $\div(\d_t)$ is smooth up to $t = 0$, we compute
\begin{align*}
	\norm{t^{-\a}\on_t( t^{\a}u)}_{L^2}^2 
		&= \norm{\on_t u + \frac \a t u}_{L^2}^2 \\
		&= \norm{\on_t u}_{L^2}^2 + \a^2 \norm{\frac u t}_{L^2}^2 + \ldr{\on_t u, \frac \a t u}_{L^2} + \ldr{\frac \a t u, \on_t u}_{L^2} \\
		&= \norm{\on_t u}_{L^2}^2 + \a^2 \norm{\frac u t}_{L^2}^2 + \ldr{\frac \a t \on_t u - \on_t\left(\frac \a t u \right), u}_{L^2} \\*
		&\qquad - \ldr{\div(\d_t) \frac \a t u, u}_{L^2} \\
		&\geq \norm{\on_t u}_{L^2}^2 + (\a + \a^2) \norm{\frac u t}_{L^2}^2 - C \a \e \norm{\frac u t}^2_{L^2} \\
		&\geq \a^2 \norm{\frac u t}_{L^2}^2,
\end{align*}
if 
\[
	\e \leq \frac 1 C.
\]
Substituting $u$ with $t^{-\a} u$ finishes the proof.
\end{proof}

\subsection{Properties of the null time function} \label{subsec: properties of t}
In order to prove the Carleman estimate, Theorem \ref{thm: Carleman estimate}, the first step is to compute asymptotic properties of the metric in terms the null time function as $t \to 0$.
Recall the canonical splits \eqref{eq: tangent bundle split} and \eqref{eq: hypersurface split}, i.e.\
\begin{align*}
	T((-\e, \e) \times N) 
		&= \R\d_t \oplus \R W \oplus E, \\
	T\left(\{\tau\} \times N \right) 
		&= \R W \oplus E.
\end{align*}
for all $\tau \in (-\e, \e)$ and that
\begin{align*}
	[\d_t, W] 
		&= 0, \\
	[\d_t, X]
		&\in E
\end{align*}
for any smooth vector field $X \in E$.
Recall also that we identify the hypersurface $\{0\} \times N$ with $N$.

\begin{lemma} \label{le: components prep}
For any smooth vector field $X$ in $E$, we have
\[
		[W, X] \in E
\]
on $(-\e, \e) \times N$.
\end{lemma}
\begin{proof}
The proof relies on our assumption that
\[
	\Ric(V, Y) = 0
\]
for all $Y \in TN$.
By the proof of Proposition \ref{prop: null time function}, there is a smooth one-form $\o$ on $N$ such that
\[
	\n_Y V = \o(Y)V
\]
for all $Y \in TN$.
Recall that $E|_{t = 0} = \ker(\o)$ and, by Proposition \ref{prop: null time function}, that
\[
	g|_{t = 0} = 
	\begin{pmatrix}
		0 & 1 & 0 \\
		1 & 0 & 0 \\
		0 & 0 & \g
	\end{pmatrix},
\]
with respect to the splitting \eqref{eq: tangent bundle split}, where $\g$ is the induced positive definite metric on $E$.

We first show that $[W, X]|_{t = 0} \in E|_{t = 0}$, i.e. that $[V, X|_{t = 0}] \in \ker(\o)$, i.e.\ that $\n_{[V, X]}V|_{t = 0} = 0$. 
We already know that $\n_{[V, X]}V|_{t = 0}$ is proportional to $V$. 
It is therefore sufficient to prove that $g(\n_{[V, X]} V, \d_t)|_{t = 0} = 0$.
Using that $\n_Y V = \o(Y)V$ for all $Y \in TN$ and $\n_VV|_{t = 0} = \frac12 V|_{t = 0}$, we have
\begin{align*}
	0 	&= \Ric(X, V)|_{t = 0} \\
		&= R(\d_t, X, V, V)|_{t = 0} + R(V, X, V, \d_t)|_{t = 0} + \tr_\g \left(R(\cdot, X, V, \cdot) \right)|_{t = 0} \\
		&= g(\n_V \n_X V, \d_t)|_{t = 0} - g(\n_X \n_V V, \d_t)|_{t = 0} - g(\n_{[V, X]} V, \d_t)|_{t = 0}  \\
		&\quad +  \tr_\g \left( g(\n_{\cdot} \n_X V, \cdot)|_{t = 0} - g(\n_X \n_{\cdot} V, \cdot)|_{t = 0} - g(\n_{[\cdot, X]} V, \cdot)|_{t = 0} \right) \\
		&= -g(\n_{[V, X]} V, \d_t)|_{t = 0}.
\end{align*}
This shows $[V, X]|_{t = 0} \in E$, as claimed.

Let $e_2, \hdots, e_n$ be a local frame of $E$, defined on $(-\e, \e) \times U$ for some open subset $U \subset N$, such that $[\d_t, e_i] = 0$.
We show that $[W, e_i] \in E$ on $(-\e, \e) \times U$, for each $i$.
The Jacobi identity implies that
\[
	[\d_t, [W, e_i]] = [W, [\d_t, e_i]] + [e_i, [W, \d_t]] = 0.
\]
Writing $[W, e_i] = f_1 W + \sum_{j = 2}^n f_j e_j$, we conclude that
\[
	0 = [\d_t, [W, e_i]] = (\d_t f_1)W + \sum_{j = 2}^n (\d_tf_j) e_j,
\]
which implies that all $f_j$ are independent of $t$.
By what we have already proven, we know that $[W, e_i]|_{t = 0} \in E$ and hence may conclude that $f_1 = 0$, which in turn implies that $[W, e_i] \in E$.
Now, for a general vector field $X = \sum_{i = 2}^n X_i e_i \in E$, we conclude that
\[
	[W, X] = \sum_{i = 2}^n (\d_W X_i)e_i + X_i [W, e_i] \in E,
\]
as claimed.
\end{proof}

We now turn to the asymptotic behaviour of the spacetime metric close to $t = 0$.
It will be convenient to use the following notation.

\begin{notation}
Let $\f$ denote any \emph{smooth} function or tensor defined on some subset 
\[
	(-\e, \e) \times U,
\]
where $U \subset N$ is an open subset. 
It will be clear from the context what type of tensor $\f$ denotes.
For the special case of smooth vector fields in $E$, it turns out convenient to use a separate notation. 
Let $\W$ denote a \emph{smooth} vector field defined on some $(-\e, \e) \times U$, such that 
\[
	\W \in E
\]
on $(-\e, \e) \times U$.
We will use the notation $\f$ and $\W$ whenever the exact form is not important, the value of $\f$ and $\W$ may change from term to term.
By \eqref{eq: tangent bundle split}, any smooth vector field $X$ on $(-\e, \e) \times U$ may be expressed as
\[
	X = \f\d_t + \f W + \W,
\]
where $\f$ here denotes some smooth functions.
If, for example, we have the additional information that $X|_{t = 0} \in TN$, then we may write (in spirit of Taylor's theorem)
\[
	X = \f t \d_t + \f W + \W
\]
to emphasise this.
\end{notation}

At this point, it might seem natural to express the metric with respect to the splitting \eqref{eq: tangent bundle split}.
As it turns out, it is far more convenient to work in a slightly more orthogonal frame. 
In the next proposition, we therefore use $\grad(t)$ instead of $\d_t$.

\begin{prop}[The components of the metric] \label{prop: metric in time function}
There is an $\e > 0$, such that $\grad(t)$ is transversal to the hypersurfaces $\{t\} \times N$ for $t \in (-\e, \e) \backslash \{0\}$ and the spacetime metric is given by 
\begin{align*}
	g &= 
	\begin{pmatrix}
		t & 0 & 0 \\
		0 & -t & 0 \\
		0 & 0 & \bar g
	\end{pmatrix} + \begin{pmatrix}
		\f t^2 & 0 & 0 \\
		0 & \f t^2 & \f t^2 \\
		0 & \f t^2 & 0
	\end{pmatrix}, \\
	g^{-1} &= 
	\begin{pmatrix}
		\frac1t & 0 & 0 \\
		0 & -\frac1t & 0 \\
		0 & 0 & \bar g^{-1}
	\end{pmatrix} + \begin{pmatrix}
		\f & 0 & 0 \\
		0 & \f & \f t \\
		0 & \f t & \f t^2
	\end{pmatrix},
\end{align*}
with respect to the splitting 
\begin{equation} \label{eq: splitting TN e}
	T((-\e, \e) \times N)|_{t \neq 0} = \R\grad(t) \oplus \R W \oplus E.
\end{equation}
Here, $\bar g$ is a smooth family of positive definite metrics on $E$.
Moreover, we have
\begin{align}
	\grad(t) 
		&= (t+\f t^2)\d_t + (1+\f t)W + t Z, \label{eq: grad(t)} \\
	g(W, \d_t) 
		&= 1. \nonumber
\end{align}
\end{prop}

Equation \eqref{eq: grad(t)} implies that $\grad(t)|_{t = 0} \in TN$.
Therefore the splitting  \eqref{eq: splitting TN e} does not extend to $t = 0$.

\begin{proof}
Recall that $[\d_t, W] = 0$ by construction.
By Proposition \ref{prop: null time function}, we know that $g(W, \d_t)|_{t = 0} = g(V, \d_t)|_{t = 0} = 1$.
Since
\begin{align*}
	\d_t g(W, \d_t) 
		&= g(\n_tW, \d_t) + g(W, \n_t \d_t) \\
		&= \frac12 \d_Wg(\d_t, \d_t) \\
		&= 0,
\end{align*}
we conclude that $g(W, \d_t) = 1$.
Using this and our assumption $\n_V V = \frac 12 V$ along $N$, we compute that
\begin{align*}
	\d_t g(W, W)|_{t = 0} 
		&= 2 g(\n_t W, W)|_{t = 0} \\
		&= - 2g(\d_t, \n_W W)|_{t = 0} \\*
		&= - 2g(\d_t, \n_V V)|_{t = 0} \\*
		&= - 1.
\end{align*}
Recall that $g(W, W)|_{t = 0} = g(V, V)|_{t = 0} = 0$.
Compactness of $N$ and Taylor's theorem imply therefore
\[
	g(W, W) = -t + \f t^2,
\]
as claimed.

Note that
\[
	g(\grad(t), X) = dt(X) = 0
\]
for all vectors $X$ tangent to $\{t\} \times N$ for any $t \in (-\e, \e)$.
In other words, $\grad(t)$ is orthogonal to any hypersurface $\{t\} \times N$. 
Since $N = \{0\} \times N \subset M$ is lightlike, it follows that $\grad(t)|_{t = 0} \in TN$ is lightlike and therefore $\grad(t)|_{t = 0} = fV$ for some smooth function $f$ on $N$.
Using $g(V, \d_t)|_{t = 0} = 1$, we conclude that
\[
	f = g(\grad(t), \d_t)|_{t = 0} = 1.
\]
Let now $\psi$ be the smooth function such that $\grad(t) - \psi \d_t \in T(\{t\} \times N)$, for all $t \in (-\e, \e)$.
We already know that $\psi|_{t = 0} = 0$ and we compute
\begin{align*}
	\d_t\psi|_{t = 0}
		&= \d_t g(\grad(t), \psi\d_t)|_{t = 0} \\*
		&= \d_t g(\grad(t), \grad(t))|_{t = 0} \\*
		&= 2 g(\n_t \grad(t), W)|_{t = 0} \\*
		&= 2 \d_t g(\grad(t), W)|_{t = 0} - 2 g(W, \n_tW)|_{t = 0} \\
		&= - \d_t g(W, W)|_{t = 0} \\
		&= 1.
\end{align*}
Taylor's theorem implies that $\psi = t + \f t^2$, which yields the expression \eqref{eq: grad(t)} for $\grad(t)$.
We conclude that
\[	
	g(\grad(t), \grad(t)) = \psi = t + \f t^2.
\]

Since $g(W, X)|_{t = 0} = g(V, X)|_{t = 0} = 0$ for any smooth vector field $X \in E$, it only remains to show that $\d_tg(W, X)|_{t = 0} = 0$.
By Lemma \ref{le: components prep}, we know that $[W, X], [\d_t, X] \in E$.
Using this, we compute
\begin{align*}
	\d_tg(W, X)|_{t = 0} 
		&= g(\n_t W, X)|_{t = 0} + g(W, \n_tX)|_{t = 0} \\
		&= g(\n_W \d_t, X)|_{t = 0} + g(W, \n_X \d_t)|_{t = 0} \\
		&= -g(\d_t, \n_W X)|_{t = 0} - g(\n_X W, \d_t)|_{t = 0} \\
		&= -g(\d_t, [W, X])|_{t = 0} - 2g(\n_X W, \d_t)|_{t = 0} \\
		&= 0,
\end{align*}
since $\d_t|_{t = 0} \perp E|_{t = 0}$.

This completes the computation of the spacetime metric $g$.
In order to compute the asymptotics for the inverse, let us write
\[
	g = A(t) + B(t),
\]
where 
\[
	A(t) = \begin{pmatrix}
		t & 0 & 0 \\
		0 & -t & 0 \\
		0 & 0 & \bar g
	\end{pmatrix}, \quad B(t) = \begin{pmatrix}
		\f t^2 & 0 & 0 \\
		0 & \f t^2 & \f t^2 \\
		0 & \f t^2 & 0
	\end{pmatrix}.
\]
Shrinking $\e$ if necessary, we can ensure that $A^{-1}(t)B(t)$ is sufficiently small for the following computation.
\begin{align*}
	g^{-1}
		&= \left(\id + A^{-1}B\right)^{-1}A^{-1} \\
		&= \sum_{n = 0}^\infty \left(-A^{-1}B\right)^nA^{-1} \\
		&= A^{-1} - A^{-1}BA^{-1} + \Omega t^2,
\end{align*}
where $\Omega$ is a matrix with coefficients which are smooth on $(-\e, \e) \times N$.
Carrying out the matrix multiplication completes the proof.
\end{proof}

Though we did not assume that $N$ was a Cauchy horizon, we have the following consequence of Proposition \ref{prop: metric in time function}:

\begin{cor} \label{cor: N is Cauchy horizon}
There is an $\e > 0$ such that any hypersurface 
\[
	\{\tau\} \times N \subset (-\e, 0) \times N
\]
is an acausal hypersurface in $M$, for which $N$ is the future Cauchy horizon.
Moreover, if $N$ is the future Cauchy horizon of some other closed acausal hypersurface $\S \subset M$, then $D(\S) = D(\{\tau\} \times N)$.
\end{cor}
\begin{proof}
We define $\e > 0$ small enough to ensure that $g(W, W)|_{(-\e, 0) \times N} > 0$ and $g(\grad(t), \grad(t))|_{(-\e, 0) \times N} < 0$.
It follows from Proposition \ref{prop: metric in time function} that such an $\e$ exists and that the hypersurfaces $\{\tau\} \times N$ with $\tau \in (-\e, 0)$ are spacelike.
Hence $t$ is a strictly monotone function along causal curves in $(-\e, 0) \times N$, which implies that all hypersurfaces $\{\tau\} \times N$ are acausal for all $\tau \in (-\e, 0)$.

By compactness of $N$, any inextendible causal curve through $\S$ intersects $\{\tau\} \times N$ for all $\tau \in (-\e, 0)$.
It follows that $(-\e, 0) \times N$ is a globally hyperbolic spacetime with Cauchy hypersurface $\{\tau\}$.
The future boundary of $(-\e, 0) \times N$ is the Cauchy horizon $N$, which we here identify with $\{0\} \times N$.

The fact that $D(\S) = D(\{\tau\} \times N)$ follows by \cite{Petersen2018}*{Prop.\ 3.1}.
\end{proof}

Corollary \ref{cor: N is Cauchy horizon} will be useful in proving Theorem \ref{thm: Killing Extension bh}, since we may now apply \cite{PetersenRacz2018}*{Thm. 1.2} to vacuum spacetimes without further assumptions on $N$ than those in Assumption \ref{ass: N compact}.

The operator $\frac1t \on_W$ will turn out to play an essential role in the Carleman estimate. 
Using Proposition \ref{prop: metric in time function}, we may compute its formal adjoint close to $t = 0$.

\begin{cor} \label{cor: formal adjoint V}
The (formal) adjoint is given by
\[
	\left( \frac1t \on_W \right)^* = - \left( \frac1t \on_W \right) + \f.
\]
\end{cor}
\begin{proof}
By Proposition \ref{prop: metric in time function}, we have
\begin{align*}
	\div(W)|_{t = 0}
		&= g(\n_t W, V)|_{t = 0} + g(\n_VV, \d_t)|_{t = 0} + \sum_{i, j = 2}^n g^{ij}g(\n_{e_i}V, e_j)|_{t = 0} \\
		&= \d_Vg(\d_t, V)|_{t = 0} \\
		&= 0. 
\end{align*}
By Taylor's theorem, we conclude that $\div(W) = \f t$.
Using $\on a = 0$, this implies
\begin{align*}
	\left( \frac1t \on_W \right)^* 
		&= - \left( \frac1t \on_W \right) - \div\left( \frac Wt \right) \\
		&= - \left( \frac1t \on_W \right) + \frac1{t^2} g(\grad(t), W) - \frac1t \div(W) \\
		&= - \left( \frac1t \on_W \right) + \f,
\end{align*}
as claimed.
\end{proof}

We may now compute the Hessian of the null time function close to $t = 0$.

\begin{prop}[Hessian of the null time function] \label{prop: Hessian}
With respect to the splitting 
\begin{equation}
	T((-\e, \e) \times N)|_{t \neq 0} = \R\grad(t) \oplus \R W \oplus E, \label{eq: splitting}
\end{equation}
the Hessian of the null time function is given by
\[
	\Hess(t) = 	\begin{pmatrix} 
						\frac t2 & 0 & 0 \\
						0 & -\frac t2 & 0 \\
						0 & 0 & \f t
					\end{pmatrix} + Bt^2,
\]
where the coefficients of the $2$-tensor $B$ with respect to \eqref{eq: splitting} are smooth up to $t = 0$ on $(-\e, \e) \times N$.
\end{prop}

Note in particular that $B(\grad(t), \grad(t))$ is smooth up to $t = 0$. 
We do not claim that for example $B(\d_t, \d_t)$ is smooth up to $t = 0$.

\begin{remark} \label{rmk: pseudo-convex}
Proposition \ref{prop: Hessian} actually reveals a certain relation between our unique continuation theorem, Theorem \ref{thm: unique continuation}, and H\"{o}rmander's local unique continuation theorem \cite{Hormander1985}*{Thm.\ 28.3.4}.
H\"{o}rmander's pseudo-convexity (for wave equations) is the condition that $\Hess(s)(X, X)|_{s = 0} < 0$ for all lightlike vectors $X$ and a hypersurface defining function $s$.
Note that his condition is \emph{directed}, i.e.\ substituting $s$ with $-s$ gives an opposite condition.
In our setting, Proposition \ref{prop: metric in time function} implies for $t < 0$ that the hypersurfaces $\{t\} \times N$ are \emph{spacelike} and therefore pseudo-convex. 
The lightlike hypersurface $\{0\} \times N$ is \emph{not} pseudo-convex, by Proposition \ref{prop: Hessian}.
Finally, we claim that Proposition \ref{prop: Hessian} implies that $\{t\} \times N$ is pseudo-convex for small enough $t > 0$.
For this, let $X$ be a lightlike vector field such that $X|_{\{t\} \times N} \in T(\{t\} \times N)$ for all $t \in (-\e, \e)$ and normalize it such that $X|_{t = 0} = V$.
Then we have $[\d_t, X] \in T(\{t\} \times N)$ for all $t \in (-\e, \e)$.
We may now use Proposition \ref{prop: Hessian} to compute that
\begin{align*}
	\d_t\left( \Hess(t)(X, X)\right)|_{t = 0} 
		&= \L_t(\Hess(t))(V, V)|_{t = 0} + 2 \Hess(t)([\d_t, X], V)|_{t = 0} \\
		&= \d_t \left( \Hess(t)(W, W) \right)|_{t = 0} - 2 \Hess([\d_t, W], V)|_{t = 0} \\
		&= - \frac12.
\end{align*}
This shows that for small enough $t > 0$, the hypersurfaces $\{t\} \times N$ are pseudo-convex.

To sum up, we have shown that the hypersurfaces $\{t\} \times N$ are \emph{pseudo-convex} in the sense of H\"{o}rmander, if $t \neq 0$ and $\abs{t}$ is sufficiently small, but is \emph{not} pseudo-convex if $t = 0$.
The crucial part in the above computation is the equality
\[
	\L_t(\Hess(t))(V, V)|_{t = 0} = - \frac12.
\]
This is directly related to our assumption that the surface gravity $\kappa$ i non-zero.
Indeed, let us consider the spacetimes in Example \ref{ex: van surf grav}, where $V = \d_x|_{t = 0}$, for which we have
\[
	\L_t \left(\Hess(t)\right)(\d_x, \d_x)|_{t = 0} = 0,
\]
if $m \geq 2$, i.e.\ when the surface gravity $\kappa$ vanishes.
On the other hand, one checks that all hypersurfaces $\{t\} \times N$ with $t \neq 0$ are pseudo-convex in the sense of H\"{o}rmander.
Our non-degeneracy assumption $\kappa \neq 0$ can therefore be seen as ensuring that the \emph{pseudo-convexity is violated only to first order at $t = 0$}, as opposed to when $\kappa = 0$, where the pseudo-convexity is violated to higher order and unique continuation indeed is false by Example \ref{ex: van surf grav}.
\end{remark}

\begin{proof}[Proof of Proposition \ref{prop: Hessian}]
Proposition \ref{prop: metric in time function} and especially equation \eqref{eq: grad(t)} are the essential ingredients in the proof. 
For any smooth vector field $X$, we have
\begin{align*}
	\Hess(t)(X, \grad(t))
		&= g(\n_X \grad(t), \grad(t)) \\
		&= \frac12 \d_X g(\grad(t), \grad(t)) \\
		&= \frac12 \d_X (t + \f t^2).
\end{align*}
It follows that
\begin{align*}
	\Hess(t)(\grad(t), \grad(t))
		&= \frac t2 + \f t^2, \\
	\Hess(t)(W, \grad(t))
		&= \f t^2, \\
	\Hess(t)(Y, \grad(t))
		&= \f t^2,
\end{align*}
for any smooth vector field $Y$ in $E$.
Note that
\[
	[W, \grad(t)] = \f t^2 \d_t + \f t W + t \W.
\] 
We get
\begin{align*}
	\Hess(t)(W, W) 
		&= g(\n_W \grad(t), W) \\*
		&= g([W, \grad(t)], W) + \frac12 \d_{\grad(t)}g(W, W) \\*
		&= \f t^2 g(\d_t, W) + \f t g(W, W) + t g(Z, W) + \frac12 \d_{\grad(t)}(-t + \f t^2) \\
		&= - \frac t 2 + \f t^2.
\end{align*}
We also get
\begin{align*}
	\Hess(t)(Y, W)
		&= g(\n_Y\grad(t), W) \\*
		&= - g(\grad(t), \n_Y W) \\*
		&= (-t + \f t^2)g(\d_t, \n_Y W) + (-1 + \f t)g(W, \n_Y W) + t g(Z, \n_Y W) \\
		&= -\frac12 \d_Y g(W, W) + \f t^2 \\
		&= -\frac12 \d_Y(-t + \f t^2) + \f t^2 \\
		&= \f t^2,
\end{align*}
where $Y$ is a smooth vector field in $E$.
The last component of the Hessian is verified just by noting that
\begin{align*}
	\Hess(t)(Y_1, Y_2)|_{t = 0}
		&= g(\n_{Y_1}\grad(t), Y_2)|_{t = 0} \\
		&= g(\n_{Y_1}V, Y_2)|_{t = 0} \\
		&= 0.
\end{align*}
This completes the proof.
\end{proof}

Let us briefly explain the main role of Proposition \ref{prop: Hessian} in the proof of Theorem \ref{thm: unique continuation}.
Recall from Subsection \ref{subsec: Strategy}, in particular equation \eqref{eq: commutator}, that it will be crucial to compute the commutator
\[
	\left[\bBox, \frac1t \on_{\grad(t)}\right] = - g(\L_{\frac{\grad(t)}t}g, \on^2) + l.o.t.
\]
The leading order term in this expression can be computed using Proposition \ref{prop: Hessian}.
Remarkably, we get the following simple form:
\begin{cor} \label{cor: Lie derivative} 
With respect to the splitting 
\begin{equation}
	T((-\e, \e) \times N)|_{t \neq 0} = \R\grad(t) \oplus \R W \oplus E, \label{eq: splitting cor}
\end{equation}
we have
\[
	\L_{\frac{\grad(t)}t}g = 	\begin{pmatrix} 
						-1 & 0 & 0 \\
						0 & -1 & 0 \\
						0 & 0 & \f
					\end{pmatrix} + Bt,
\]
where the coefficients of the $2$-tensor $B$ with respect to the splitting \eqref{eq: splitting cor} are smooth on $(-\e, \e) \times N$.
\end{cor}
\begin{proof}
For any vector fields $X, Y$, we have
\begin{align*}
	\L_{\frac{\grad(t)}t}g(X, Y) 
		&= g\left(\n_X \left(t^{-1}\grad(t)\right), Y \right) + g\left(\n_Y \left(t^{-1}\grad(t)\right), X \right) \\
		&= \frac 2t \Hess(t)(X, Y) - \frac{2}{t^2}dt(X)dt(Y).
\end{align*}
It is therefore clear that 
\[
	\L_{\frac{\grad(t)}t}g(X, Y) = \frac2t \Hess(t)(X, Y)
\]
if either $X$ or $Y$ is tangent to the level hypersurfaces $\{t\} \times N$.
The only term we need to compute is when $X = Y = \grad(t)$.
By Proposition \ref{prop: metric in time function}, it follows that
\[
	dt(\grad(t)) = t + \f t^2.
\]
The statement now follows from Proposition \ref{prop: Hessian}.
\end{proof}

The following is an almost immediate consequence of Proposition \ref{prop: Hessian}.
\begin{cor} \label{cor: wave on t}
For all $\b \in \R$, we have
\[
	t^{-\b}\Box(t^\b) = - \frac{\b^2}{t} + \f \b + \f \b^2.
\]
\end{cor}
\begin{proof}[Proof of Corollary \ref{cor: wave on t}]
We compute
\begin{align*}
	t^{-\b} \Box (t^\b)
		&= -t^{-\b}\div\left( \b t^{\b-1} \grad(t)\right) \\
		&= - \frac{\b}{t} \tr_g(\Hess(t)) - \frac{(\b-1)\b}{t^2}g(\grad(t), \grad(t)). 
\end{align*}
Proposition \ref{prop: metric in time function} and Proposition \ref{prop: Hessian} now imply the statement.
\end{proof}

Using this, we make the following useful observation.

\begin{cor} \label{cor: formal adjoint grad t}
The (formal) adjoint is given by
\[
	\left(\frac{1}t \on_{\grad(t)} \right)^* 
		= - \frac{1}t \on_{\grad(t)} + \f.
\]
\end{cor}
\begin{proof}[Proof of Corollary \ref{cor: formal adjoint grad t}]
Using $\on a = 0$, Proposition \ref{prop: metric in time function} and Corollary \ref{cor: wave on t} imply
\begin{align*}
	\left(\frac1t \on_{\grad(t)} \right)^* 
		&= - \frac1t \on_{\grad(t)} - \div\left(\frac{\grad(t)}{t} \right) \\
		&= - \frac1t \on_{\grad(t)} + \frac1{t^2}g(\grad(t), \grad(t)) + \frac1t \Box (t) \\
		&= - \frac1t \on_{\grad(t)} + \frac1t + \f - \frac1t + \f \\
		&= - \frac1t \on_{\grad(t)} + \f. \qedhere
\end{align*}
\end{proof}

The following corollary also turns out to be important later. 

\begin{cor} \label{cor: VV grad grad}
We have
\begin{align*}
	\n_{\grad(t)}\grad(t) 
		&= \frac12 \grad(t) + \f t^2 \d_t + \f t W + t Z, \\
	\n_W W
		&= \frac12 \grad(t) + \f t^2 \d_t + \f t W + t Z.
\end{align*}
\end{cor}
\begin{proof}
Note that
\[
	\n_{\grad(t)}\grad(t) = \sum_{\b, \gamma = 0}^n\Hess(t)(\grad(t), e_\b)g^{\b \gamma}e_\gamma.
\]
Proposition \ref{prop: metric in time function} and Proposition \ref{prop: Hessian} now imply the first statement.
By Proposition \ref{prop: Hessian}, we have
\begin{align*}
	g(\n_W W, \grad(t))
		&= - \Hess(t)(W, W) \\
		&= \frac t2 + \f t^2, \\
	g(\n_W W, W)
		&= \frac12 \d_W g(W, W) \\
		&= \f t^2, \\
	g(\n_W W, X)
		&= \f t,
\end{align*}
for any smooth vector field $X \in E$.
Proposition \ref{prop: metric in time function} now implies the second statement.
\end{proof}

We conclude with the following observation.

\begin{cor} \label{cor: wave on grad(t) over t}
The vector field $t \Box\left(\frac{\grad(t)}t\right)$ is smooth on $(-\e, \e) \times N$ and 
\[
	t \Box\left(\frac{\grad(t)}t\right)|_{t = 0} \in TN.
\]
\end{cor}

Let us emphasise that $\Box$ in Corollary \ref{cor: wave on grad(t) over t} is defined using the Levi-Civita connection with respect to the \emph{indefinite} metric $g$, as opposed to $\bBox$ on the vector bundle $F$, which was defined using a connection which was compatible with the positive definite metric $a$.

\begin{proof}
By Corollary \ref{cor: wave on t} and Corollary \ref{cor: VV grad grad}, we obtain
\begin{align*}
	t \Box\left(\frac{\grad(t)}t\right)
		&= \Box \grad(t) + t \Box\left(\frac1t \right)\grad(t) - 2t \n_{\grad\left(\frac1t \right)}\grad(t) \\
		&= \Box \grad(t) - \frac1t \grad(t) + \frac2t\n_{\grad(t)}\grad(t) + \f \grad(t) \\
		&= \Box \grad(t) + \f \grad(t) + \f W + Z.
\end{align*}
By the Weitzenböck formula we have
\[
	\Box \grad(t) = \grad(\Box t) - \Ric(\grad(t)).
\]
Thus, by Corollary \ref{cor: wave on t} we conclude
\begin{align*}
	g(\Box \grad(t), V)|_{t = 0}
		&= \d_V \Box t|_{t = 0} - \Ric(\grad(t), V)|_{t = 0} \\
		&= \d_V(-1 + \f t)|_{t = 0} - \Ric(V, V)|_{t = 0} \\
		&= 0.
\end{align*}
This proves that $\Box \grad(t)|_{t = 0} \in TN$ and hence
\[
	t \Box\left(\frac{\grad(t)}t\right)|_{t = 0} \in TN,
\]
as claimed.
\end{proof}

\subsection{Proof of the Carleman estimate} \label{subsec: proof Carleman}

In this subsection we prove Theorem \ref{thm: Carleman estimate}.
We first rewrite Theorem \ref{thm: Carleman estimate} in terms of the conjugate operator. 

\begin{definition}[The conjugate wave operator]
For any $\a \in \N$, define
\[
	\bBox_\a(u) := t^{-\a}\bBox (t^\a u)
\]
for any $u \in C^\infty((-\e, \e) \times N, F)$.
\end{definition}

\begin{remark}
By substituting $u$ with $t^\a u$, we note that Theorem \ref{thm: Carleman estimate} is equivalent to the following statement:

\emph{Let $M$ and $N$ satisfy Assumption \ref{ass: N compact} and let $F \to M$ be a real or complex vector bundle.
There are constants $\e, \a_0, C > 0$ such that 
\begin{equation} \label{eq: version of Carleman estimate}
	\norm{\bBox_\a u}_{L^2} \geq \sqrt \a C \norm{u}_{H^1_\a}
\end{equation}
for all $u \in C_*^\infty((-\e, \e) \times N, F)$ and all integers $\a \geq \a_0$.}
\end{remark}

The remainder of this subsection is devoted to proving the estimate \eqref{eq: version of Carleman estimate}.
We split $\bBox_\a$ into formally self-adjoint and anti-self-adjoint parts $\bBox_\a^s$ and $\bBox_\a^a$ respectively, i.e.\
\begin{align*}
	\bBox_\a^s 
		&:= \frac{\bBox_\a + (\bBox_\a)^*}2, \\
	\bBox_\a^{a}
		&:= \frac{\bBox_\a - (\bBox_\a)^*}2.
\end{align*}
It follows that $(\bBox^s_\a)^* = \bBox^s_\a$ and $(\bBox^a_\a)^* = - \bBox^a_\a$.
Equation \eqref{eq: Box alpha} implies that
\begin{align}
	\norm{\bBox_\a u}^2_{L^2} = \norm{\bBox_\a^s u}^2_{L^2} + \norm{\bBox_\a^a u}^2_{L^2} + \ldr{[\bBox_\a^s, \bBox_\a^a] u, u}_{L^2}.
\end{align}
The proof of Theorem \ref{thm: Carleman estimate} consists of computing these terms using the results of the previous subsection and proving suitable lower bounds.

\begin{lemma}[The first estimates] \label{le: first estimates}
There are constants $\e_0, \a_0, C > 0$, such that 
\begin{align*}
	\norm{\bBox_\a^s u}_{L^2} 
		&\geq \norm{\left(\bBox - \frac{\a^2}t \right) u}_{L^2} - \e \a C \norm{u}_{H^1_\a}, \\
	\norm{\bBox_\a^a u}_{L^2}
		&\geq 2 \a \norm{\frac1t \on_{\grad(t)}u}_{L^2} - \e C \norm{u}_{H^1_\a}, \\
	\ldr{[\bBox_\a^s, \bBox_\a^a]u, u}_{L^2}
		&\geq -2\a \ldr{\left[\bBox, \frac1t\on_{\grad(t)}\right]u, u}_{L^2} + 2\a \norm{\frac{\a u}t}_{L^2}^2 - \e \a C \norm{u}_{H^1_\a}^2,
\end{align*}
for all $u \in C_*^\infty((-\e, \e) \times N, F)$, for any $\e \in (0, \e_0)$ and any integer $\a \geq \a_0$.
\end{lemma}
\begin{proof}
By Corollary \ref{cor: wave on t}, we first observe
\begin{align*}
	\bBox_\a 
		&= \bBox - \frac{2\a}t \on_{\grad(t)} + t^{-\a}\bBox(t^\a), \\
		&= \bBox - \frac{2\a}t \on_{\grad(t)} - \frac {\a^2}t + p_2(\a),
\end{align*}
where $p_m(\a)$ is some polynomial in $\a$ of order $m$ with smooth coefficients. 
The exact coefficients of $p_m$ will not be important and might change from term to term.
By Corollary \ref{cor: formal adjoint grad t}, we conclude that
\begin{align*}
	\bBox_\a^s 
		&= \bBox - \frac{\a^2}t + p_2(\a), \\
	\bBox_\a^a
		&= - \frac{2\a}t \on_{\grad(t)} + p_1(\a).
\end{align*}
Let $C > 0$ denote some constant which may change from term to term.
Since
\begin{align*}
	- \norm{p_2(\a) u}_{L^2} 
		&\geq - \e \a C \norm{\frac {\a u}t}_{L^2} \\*
		&\geq - \e \a C \norm{u}_{H^1_\a}, \\
	- \norm{p_1(\a) u}_{L^2}
		&\geq - \e C \norm{u}_{H^1_\a},
\end{align*} 
for large enough $\a$, the first two estimates are clear.
By equation \eqref{eq: grad(t)}, we get
\begin{align}
	[\bBox_\a^s, \bBox_\a^a]
		&= -2\a\left[\bBox, \frac1t \on_{\grad(t)}\right] + \frac{2\a^3}t \left[\frac1t, \on_{\grad(t)}\right] + \frac{p_3(\a)}t + [\bBox, p_1(\a)], \nonumber \\
		&= -2\a\left[\bBox, \frac1t \on_{\grad(t)}\right] + \frac{2\a^3}{t^2} + \frac{p_3(\a)}t - 2 \on_{\grad(p_1(\a))}. \label{eq: commutator in proof}
\end{align}
We first observe that
\begin{align*}
	\ldr{\frac{p_3(\a)}t u, u}_{L^2} \geq - \e \a C \norm{\frac {\a u} t}_{L^2}^2 \geq - \e \a C \norm{u}_{H^1_\a}^2.
\end{align*}
To estimate the last term in equation \eqref{eq: commutator in proof}, note that by equation \eqref{eq: grad(t)} we may schematically write
\begin{align*}
	\grad(p_1(\a)) 
		&= p_1(\a)\d_t + p_1(\a)W + p_1(\a)Z \\
		&= p_1(\a) \frac{\grad(t)}t + p_1(\a) \frac Wt + p_1(\a) Z.
\end{align*}
This implies
\begin{align*}
	-2\ldr{\on_{\grad(p_1(\a))}u, u}_{L^2}
		&\geq - C \abs{\ldr{\frac1t \on_{\grad(t)}u, \a u}_{L^2}} - C \abs{\ldr{\frac1t \on_Wu, \a u}_{L^2}} \\*
		&\qquad - C \abs{\ldr{\on_Zu, \a u}_{L^2}} \\
		&\geq - \e \a C \norm{u}_{H^1_\a}^2,
\end{align*}
which completes the proof of the lemma.
\end{proof}

Since we will choose $\e$ very small, the terms of the form $-\e \a C\norm{u}_{H^1_\a}$ will be small compared to the rest.
From Lemma \ref{le: first estimates}, the importance of computing the commutator
\begin{equation} \label{eq: commutator grad}
	\left[\bBox, \frac1t\on_{\grad(t)}\right]
\end{equation}
is now clear.
For this, note the following lemma:
\begin{lemma}\label{le: comm X}
For any smooth vector field $X$, we have
\begin{align*}
	[\bBox, \on_X] 
		& = - g(\L_X g, \on^2) + \on_{\Box X - \Ric(X)} \\
		& \qquad - \sum_{\b, \gamma = 0}^n g^{\b \gamma}\left[2 \bR(e_\b, X) \on_{e_{\gamma}} + (\on_{e_\b}\bR)(e_\gamma, X) + \bR(e_\b, \n_{e_\gamma}X)\right],
\end{align*}
where $\bR$ is the curvature tensor associated to $\on$, considered as a homomorphism field from $TM \otimes TM \otimes F$ to $F$.
Here,
\[
	g(\L_X g, \on^2) := \sum_{i,j,k,l = 0}^ng^{ij}g^{kl}\L_Xg(e_i, e_k)\on^2_{e_j, e_l},
\]
with respect to some local frame.
\end{lemma}
\begin{proof}
This is a routine computation.
\end{proof}

Combining Corollary \ref{cor: Lie derivative} and Lemma \ref{le: comm X}, we may now compute the commutator \eqref{eq: commutator grad}.
For this, the following definition is convenient:
\begin{definition} \label{def: H-operator}
We say that $Q_m$ is an \emph{$N$-differential operator} of order $m$ if we may locally express $Q_m$ as a sum of 
\[
	A \circ \on_{X_1} \hdots \on_{X_k},
\]
for $k \leq m$, where $X_1, \hdots, X_k$ are vector fields satisfying
\[
	X_j|_{t = 0} \in TN,
\]
and $A$ is a \textbf{smooth} endomorphism of $F$. 
\end{definition}

\noindent
For example, $\on_{\grad(t)}$ is an $N$-differential operator of first order, whereas $\on_t$ is not.

\begin{lemma} \label{le: hard commutator}
We have
\begin{align}
	\left[\bBox, \frac1t\on_{\grad(t)}\right] 
		&= \left( \frac1t \on_{\grad(t)}\right)^2 + \left( \frac1t \on_W \right)^2 \nonumber \\
		&\qquad + \frac \f t \on_{\grad(t)} \on_{\grad(t)} + \frac \f t \on_W \on_W + \frac \f t \on_W\on_{\grad(t)} \nonumber \\
		&\qquad + \frac1t Q_1 + Q_2, \label{eq: commutator explicit}
\end{align}
where $Q_1$ and $Q_2$ are $N$-differential operators of first and second order, respectively.
\end{lemma}
\begin{proof}
Applying Lemma \ref{le: comm X} for $t \neq 0$ with 
\[
	X := \frac{\grad(t)}t,
\]
we get
\begin{align}
	\left[\bBox, \frac1t \on_{\grad(t)}\right] 
	&= - g\left(\L_{\frac{\grad(t)}t} g, \on^2\right) + \frac1t \on_{t \Box \left(\frac{\grad(t)}t \right) - \Ric(\grad(t))} \nonumber \\*
	&\qquad - \frac2t \sum_{\b, \gamma = 0}^n g^{\b \gamma} \bR \left(e_\b, \grad(t)\right) \on_{e_{\gamma}} \nonumber \\*
	&\qquad - \frac1t \sum_{\b, \gamma = 0}^n g^{\b \gamma}(\on_{e_\b}\bR)\left(e_\gamma, \grad(t)\right) \nonumber \\*
	&\qquad - \sum_{\b, \gamma = 0}^ng^{\b \gamma} \bR\left(e_\b, \n_{e_\gamma}\left(\frac{\grad(t)}t\right)\right). \label{eq: singular commutator}
\end{align}

By Corollary \ref{cor: wave on grad(t) over t} and since 
\[
	g(\Ric(\grad(t)), V)|_{t = 0} = \Ric(V, V)|_{t = 0} = 0,
\]
we may write
\begin{align*}
	t\Box\left( \frac{\grad(t)}t \right) - \Ric(\grad(t))
		&= \f t \d_t + \f W + Z \\
		&= \f \grad(t) + \f W + Z
\end{align*}
which implies that the second term of equation \eqref{eq: singular commutator} is of the claimed form.
For the third term of equation \eqref{eq: singular commutator}, choose a local frame $e_0 := \grad(t), e_1 := W, e_2, \hdots, e_n$, with $e_2, \hdots, e_n \in E$.
Using that $\bR(\grad(t), \grad(t)) = 0$, we have
\begin{align*}
	- \frac2t \sum_{\b, \gamma = 0}^n 
		&g^{\b \gamma} \bR(e_\b, \grad(t)) \on_{e_{\gamma}} 
		= - \frac 2t \sum_{\b, \gamma = 1}^n g^{\b \gamma} \bR(e_\b, \grad(t)) \on_{e_{\gamma}},
\end{align*}
which is of the claimed form, since $e_1, \hdots, e_n$ are all tangent to the hypersurfaces $\{t\} \times N$.
The fourth term in \eqref{eq: singular commutator} already is of the claimed form.
Finally, the fifth term of \eqref{eq: singular commutator} is computed as
\begin{align*}
	\sum_{\b, \gamma = 0}^ng^{\b \gamma} \bR\left(e_\b, \n_{e_\gamma}\left(\frac{\grad(t)}t\right)\right) 
		&= - \frac1{t^2} \bR\left(\grad(t), \grad(t)\right) \\*
		& \qquad + \frac1t \sum_{\b, \gamma = 0}^ng^{\b \gamma} \bR\left(e_\b, \n_{e_\gamma}\grad(t)\right) \\
		&= \frac1t \sum_{\b, \gamma = 0}^ng^{\b \gamma} \bR\left(e_\b, \n_{e_\gamma}\grad(t)\right),
\end{align*}
which is of the form claimed in the lemma. 

We now turn to the first term on the right hand side in equation \eqref{eq: singular commutator}. 
By Proposition \ref{prop: metric in time function} and Corollary \ref{cor: Lie derivative} one concludes that with respect to the splitting 
\begin{equation*}
	T((-\e, \e) \times N)|_{t \neq 0} = \R\grad(t) \oplus \R W \oplus E,
\end{equation*}
we have
\[
	\sum_{k,l = 0}g^{ik}g^{jl}\L_{\frac{\grad(t)}t}g(e_k, e_l) = 
	\begin{pmatrix}
		- \frac1{t^2} + \frac{\f}t & \frac{\f}t & \f \\
		\frac{\f}t & -\frac1{t^2} + \frac{\f}t & \f \\
		\f & \f & \f
	\end{pmatrix}.
\]
It follows that
\begin{align}
	-g\left(\L_{\frac{\grad(t)}t}g, \on^2\right)
		&= - \sum_{k,l = 0}g^{ik}g^{jl}\L_{\frac{\grad(t)}t}g(e_k, e_l)\on^2_{e_k, e_l}\\
		&= \frac1{t^2} \on^2_{\grad(t), \grad(t)} + \frac1{t^2}\on^2_{W, W} \nonumber \\*
		&\qquad + \frac \f t \on^2_{\grad(t), \grad(t)} + \frac \f t \on^2_{W,W} + \frac \f t \on^2_{W, \grad(t)} \nonumber \\*
		&\qquad + \frac \f t \bR(W, \grad(t)) + Q_2, \label{eq: leading term commutator}
\end{align}
where $Q_2$ is an $N$-differential operator of second order.
Let us simplify this expression.
By Corollary \ref{cor: VV grad grad} and equation \eqref{eq: grad(t)}, we first note that
\begin{align*}
	\frac1{t^2} \on^2_{\grad(t), \grad(t)} + \frac1{t^2}\on^2_{W, W} 
		&= \frac1{t^2} \on_{\grad(t)}\on_{\grad(t)} - \frac1{t^2}\on_{\n_{\grad(t)}\grad(t) + \n_W W} \\*
		&\qquad + \left( \frac1t \on_W \right)^2 \\
		&= \frac1{t^2} \on_{\grad(t)}\on_{\grad(t)} - \frac1{t^2}\on_{\grad(t)} + \left( \frac1t \on_W \right)^2\\*
		&\qquad + \frac1t (\f \on_\grad(t) + \f \on_W + \on_Z) \\
		&= \left(\frac1t \on_{\grad(t)}\right)^2 + \left( \frac1t \on_W \right)^2 \\*
		&\qquad + \frac1t (\f \on_\grad(t) + \f \on_W + \on_Z),
\end{align*}
which is of the claimed form.
The last thing to note is that $\n_{\grad(t)}\grad(t)$, $\n_W W$ and $\n_W \grad(t)$ are all tangent to $N$ at $t = 0$.
Inserting these observations into equation \eqref{eq: leading term commutator} finishes the proof of the lemma.
\end{proof}

We may now prove the lower bound for the commutator:

\begin{lemma}[Improved estimate for the commutator] \label{le: commutator}
There are constants $\a_0, \e_0, C > 0$ such that 
\begin{align*}
	\ldr{[\bBox_\a^s, \bBox_\a^a]u, u}_{L^2}
		& \geq \frac{3\a}2 \norm{u}_{H^1_\a}^2 - \a C \norm{\en u}_{L^2}^2
\end{align*}
for all $u \in C_*^\infty((-\e, \e) \times N, F)$, for any $\e \in (0, \e_0)$ and any integer $\a \geq \a_0$.
\end{lemma}
\begin{proof}
By Lemma \ref{le: first estimates}, we know that the crucial term to estimate is
\[
	- 2 \a \ldr{\left[\bBox, \frac1t \on_{\grad(t)}\right]u, u}_{L^2}.
\]
Let $C > 0$ denote some constant which may change from term to term.
By Lemma \ref{le: hard commutator} combined with Corollary \ref{cor: formal adjoint V} and Corollary \ref{cor: formal adjoint grad t}, we get the estimate
\begin{align*}
	- 2 \a \ldr{\left[\bBox, \frac1t \on_{\grad(t)}\right]u, u}_{L^2}
		&\geq 2\a \norm{\frac1t \on_{\grad(t)}u}_{L^2}^2 + 2\a \norm{\frac1t \on_W u}^2_{L^2} - \e\a C\norm{u}_{H^1_\a}^2 \\*
		&\qquad - 2\a \ldr{\frac1t Q_1 u, u}_{L^2} - 2\a \ldr{Q_2 u, u}_{L^2}
\end{align*}
for any $\a \geq \a_0$ and $\e \in (0, \e)$ for $\a_0$ large enough and $\e_0$ small enough.
By Definition \ref{def: H-operator}, we know that there are smooth endomorphisms $A_0, A_1, A_2$ and $B$, such that
\[
	\frac1t Q_1 = \frac1t A_0 \circ \on_{\grad(t)} + \frac1t A_1 \circ \on_W + \frac1t A_2 \circ \on_Z + \frac1t B.
\]
We get the estimate
\begin{align*}
	-2 \a \ldr{\frac1t Q_1 u, u}_{L^2} 
		&\geq -2\a \ldr{A_2 \circ \on_Z u, \frac u t }_{L^2} - \e\a C\norm{u}_{H^1_\a}^2\\
		&\geq - C \a \norm{\en u}_{L^2}^2 - \a\left(\e +\frac{1}{\a^2}\right) C\norm{u}_{H^1_\a}^2.
\end{align*}
Similarly, we get the analogous lower bound
\[
	-2 \a \ldr{Q_2 u, u}_{L^2} \geq - C \a \norm{\en u}_{L^2}^2 - \e\a C\norm{u}_{H^1_\a}^2.
\]
We conclude that
\begin{align*}
	- 2 \a \ldr{\left[\bBox, \frac1t \on_{\grad(t)}\right]u, u}_{L^2}
		&\geq 2\a \norm{\frac1t \on_{\grad(t)}u}_{L^2}^2 + 2\a \norm{\frac1t \on_W u}^2_{L^2} \\
		&\qquad - C \a \norm{\en u}_{L^2}^2 - \a\left(\e +\frac{1}{\a^2}\right) C\norm{u}_{H^1_\a}^2.
\end{align*}
Inserting this into Lemma \ref{le: first estimates} implies
\begin{align*}
	\ldr{[\bBox_\a^s, \bBox_\a^a] u, u}_{L^2}
		&\geq 2\a \norm{\frac1t \on_{\grad(t)}u}_{L^2}^2 + 2\a \norm{\frac1t \on_W u}^2_{L^2} + 2\a \norm{\frac{\a u}t}_{L^2}^2 \\*
		&\qquad - C \a \norm{\en u}_{L^2}^2 - \a\left(\e +\frac{1}{\a^2}\right) C\norm{u}_{H^1_\a}^2 \\
		&\geq 2\a \norm{u}_{H^1_\a}^2 - C \a \norm{\en u}_{L^2}^2 - \a\left(\e +\frac{1}{\a^2}\right) C\norm{u}_{H^1_\a}^2.
\end{align*}
Choosing $\e_0$ small enough and $\a_0$ large enough yields the claim.
\end{proof}

From Lemma \ref{le: commutator}, we see that it is necessary to compensate for the term $-\a C \norm{\en u}_{L^2}^2$.
The next lemma will provide the necessary lower bound on $\norm{\bBox_\a^su}_{L^2}$:

\begin{lemma}[Improved estimate for the self-adjoint part]\label{le: Garding}
There are constants ${\e_0, \a_0, C > 0}$, such that
\begin{align*}
	\norm{\bBox_\a^s u}_{L^2} \geq \a \frac{\norm{\en u}_{L^2}^2}{\norm{u}_{H^1_\a}} - \e \a C \norm{u}_{H^1_\a}
\end{align*}
for all $u \in C_*^\infty((-\e, \e) \times N, F)$, for any $\e \in (0, \e_0)$ and any integer $\a \geq \a_0$.
\end{lemma}
\begin{proof}
By Lemma \ref{le: first estimates}, we have
\[
	\norm{\bBox_\a^s u}_{L^2} 
		\geq \norm{\left(\bBox - \frac{\a^2} t\right) u}_{L^2} - \e \a C \norm{u}_{H^1_\a}
\]
if $\a_0$ is large enough and $\e_0$ small enough.
By Proposition \ref{prop: metric in time function}, we get
\begin{align*}
	\frac 1\a & \norm{\left(\bBox - \frac{\a^2} t \right) u}_{L^2} \\
		&\geq \frac1{\a\norm{u}_{L^2}} \abs{\ldr{\left(\bBox - \frac{\a^2} t \right) u, u}_{L^2}} \\
		&\geq \frac1{\norm{u}_{H^1_\a}} \abs{\ldr{\left(\bBox - \frac{\a^2} t \right) u, u}_{L^2}} \\
		&\geq \frac 1{\norm{u}_{H^1_\a}} \left(\ldr{\on u, \on u}_{L^2} - \abs{\ldr{\a u , \frac{\a u}t }_{L^2}}\right) \\
		&\geq \frac 1{\norm{u}_{H^1_\a}} \Big( \ldr{\en u, \en u}_{L^2} - \abs{\ldr{ g^{00} \on_{\grad(t)}u,\on_{\grad(t)}u}_{L^2}} \\*
		& \qquad - \abs{\ldr{g^{11}\on_W u, \on_W u}_{L^2}} - 2\abs{\ldr{\sum_{j= 2}^ng^{1j}\on_{e_j}u, \on_W u}_{L^2}}  - \abs{\ldr{\a u , \frac{\a u}t }_{L^2}} \Big) \\
		&\geq \frac{\norm{\en u}_{L^2}^2}{\norm{u}_{H^1_\a}} - \e C \norm{u}_{H^1_\a}. \qedhere
\end{align*}
\end{proof}

We are finally ready to prove Theorem \ref{thm: Carleman estimate}.

\begin{proof}[Proof of Theorem \ref{thm: Carleman estimate}]
Equation \eqref{eq: Box alpha} says
\begin{align*}
	\frac1\a \norm{\bBox_\a u}^2_{L^2} 
		&= \frac1\a \norm{\bBox_\a^s u}^2_{L^2} + \frac1\a\norm{\bBox_\a^a u}^2_{L^2} + \frac1\a \ldr{[\bBox_\a^s, \bBox_\a^a] u, u}_{L^2}.
\end{align*}
By Lemma \ref{le: commutator} and Lemma \ref{le: Garding}, we can fix constants $\e_0, \a_0, C > 0$, such that
\begin{align*}
	\frac 1\a \ldr{[\bBox_\a^s, \bBox_\a^a]u, u}_{L^2}
		& \geq \frac32 \norm{u}_{H^1_\a}^2 - C \norm{\en u}_{L^2}^2, \\
	\frac1{\sqrt {\a}} \norm{\bBox_\a^s u}_{L^2} 
		&\geq \sqrt \a \left(\frac{\norm{\en u}_{L^2}^2}{\norm{u}_{H^1_\a}} - \e C \norm{u}_{H^1_\a}\right)
\end{align*}
for all $u \in C_*^\infty((-\e, \e) \times N, F)$, for any $\e \in (0,\e_0)$ and any $\a \geq \a_0$.
We claim that if we increase $\a_0$ and choose $\e$ small enough to satisfy
\[
	\sqrt \a_0 (1-2\e C^2) \geq \frac C {\sqrt 2},
\]
then estimate \eqref{eq: version of Carleman estimate} holds for any integer $\a \geq \a_0$.

\noindent \textbf{Case 1:} 
Assume that $\norm{u}_{H^1_\a}^2 \leq 2C \norm{\en u}_{L^2}^2$.
In this case, it follows that
\begin{align*}
	\frac1{\sqrt \a} \norm{\bBox_\a^s u}_{L^2} 
		&\geq \frac{\sqrt \a }{\norm{u}_{H^1_\a}} \left( \norm{\en u}^2_{L^2} - \e C \norm{u}_{H^1_\a}^2 \right) \\
		&\geq \frac{\sqrt \a }{\norm{u}_{H^1_\a}} \left( \norm{\en u}^2_{L^2} - 2 \e C^2 \norm{\en u}_{L^2}^2 \right) \\
		&= \frac{\norm{\en u}_{L^2}^2}{\norm{u}_{H^1_\a}} \sqrt \a (1 - 2 \e C^2) \\
		&\geq \frac{C}{\sqrt 2}\frac{\norm{\en u}_{L^2}^2}{\norm{u}_{H^1_\a}}.
\end{align*}
Equation \eqref{eq: Box alpha} implies in this case that
\begin{align*}
	\frac1\a \norm{\bBox_\a u}^2_{L^2} 
		&\geq \frac1\a \norm{\bBox_\a^s u}_{L^2}^2 + \frac1\a \ldr{[\bBox_\a^s, \bBox_\a^a] u, u}_{L^2} \\
		&\geq \frac{C^2}{2}\frac{\norm{\en u}_{L^2}^4}{\norm{u}_{H^1_\a}^2} + \frac32 \norm{u}_{H^1_\a}^2 - C \norm{\en u}_{L^2}^2 \\
		&= \norm{u}_{H^1_\a}^2 + \frac{1}{2 \norm{u}_{H^1_\a}^2}\left( C^2 \norm{\en u}_{L^2}^4 + \norm{u}_{H^1_\a}^4 - 2C\norm{\en u}_{L^2}^2 \norm{u}_{H^1_\a}^2 \right) \\
		&= \norm{u}_{H^1_\a}^2 + \frac{1}{2 \norm{u}_{H^1_\a}^2}\left( C \norm{\en u}_{L^2}^2 - \norm{u}_{H^1_\a}^2 \right)^2 \\
		&\geq \norm{u}_{H^1_\a}^2.
\end{align*}
\textbf{Case 2:}
Let us now assume the reverse inequality $\norm{u}_{H^1_\a}^2 \geq 2C \norm{\en u}_{L^2}^2$.
In this case, 
\begin{align*}
	\frac1\a \norm{\bBox_\a u}^2_{L^2} 
		& \geq \frac 1\a \ldr{[\bBox_\a^s, \bBox_\a^a]u, u}_{L^2} \\*
		& \geq \frac32 \norm{u}_{H^1_\a}^2 - C \norm{\en u}_{L^2}^2 \\
		& \geq \norm{u}_{H^1_\a}^2 + \frac12 \left(\norm{u}_{H^1_\a}^2 - 2 C \norm{\en u}_{L^2}^2 \right) \\
		& \geq \norm{u}_{H^1_\a}^2.
\end{align*}
This finishes the proof of the theorem.
\end{proof}

\subsection{Proof of unique continuation} \label{subsec: proof unique cont}

We now use Theorem \ref{thm: Carleman estimate} and Proposition \ref{prop: Carleman estimate transport} to prove Theorem \ref{thm: unique continuation}.
It is convenient to first prove the following lemma:

\begin{lemma} \label{le: trivial conseq from ass}
If $u_1 \in C_*^\infty((-\e, \e) \times N, F_1)$ and $u_2 \in C_*^\infty((-\e, \e) \times N, F_2)$ satisfy assumption \eqref{eq: pointwise bound assumption} in Theorem \ref{thm: unique continuation}, then there is a constant $C > 0$, such that 
\[
	\norm{t^{-\a}\bBox u_1}_{L^2} + \norm{t^{-\a}\on_t u_2}_{L^2} 
		\leq C \left( \norm{t^{-\a} u_1}_{H^1_\a} + \norm{t^{-\a-1}u_2}_{L^2} \right)
\]
for any $\a$. 
The constant $C$ is independent of $\a$.
\end{lemma}
\begin{proof}
In this proof we let $C > 0$ denote some constant which may change from term to term.
The equality
\[
	t^{-\a} \on_t u_1 = \on_t(t^{-\a}u_1) + \frac{1} t t^{-\a} \a u_1
\]
and assumption \eqref{eq: pointwise bound assumption} imply that
\begin{align*}
	\abs{t^{-\a}\bBox u_1} + \abs{t^{-\a}\on_t u_2} 
		&\leq \frac C{\abs{t}} \big(\abs{\on_W (t^{-\a}u_1)} + \abs{t^{-\a}u_1} + \abs{t^{-\a}u_2} \big) \\*
		&\qquad + C \big(\abs{t^{-\a}\on_t u_1} + \abs{\en(t^{-\a} u_1)}\big) \\
		&\leq \frac C{\abs{t}} \big(\abs{\on_W (t^{-\a}u_1)} +  \abs{t^{-\a}\a u_1} + \abs{t^{-\a}u_2}\big) \\*
		&\qquad + C\big(\abs{\on_t(t^{-\a} u_1)} + \abs{\en (t^{-\a}u_1)}\big).
\end{align*}
Equation \eqref{eq: grad(t)} implies
\[
	\abs{\on_t (t^{-\a}u_1)} \leq \frac C {\abs{t}} \big(\abs{\on_{\grad(t)}(t^{-\a}u_1)} + \abs{\on_W(t^{-\a}u_1)}\big) + C \abs{\en (t^{-\a}u_1)},
\]
and we get
\begin{align*}
	\abs{t^{-\a}\bBox u_1} + \abs{t^{-\a}\on_t u_2} 
		&\leq \frac C{\abs{t}} \big(\abs{\on_{\grad(t)} (t^{-\a}u_1)} + \abs{\on_W (t^{-\a}u_1)} + \abs{t^{-\a}\a u_1}\big) \\*
		&\qquad + C \abs{\en (t^{-\a}u_1)} + \frac C{\abs t}\abs{t^{-\a}u_2}.
\end{align*}
Integrating the squared inequality gives
\begin{align*} 
	\norm{t^{-\a}\bBox u_1}_{L^2} + \norm{t^{-\a}\on_t u_2}_{L^2} 
		&\leq C\norm{\frac 1t \on_{\grad(t)}(t^{-\a}u_1)}_{L^2} + C\norm{\frac1t \on_W(t^{-\a}u_1)}_{L^2} \\*
		&\qquad + C\norm{\frac{t^{-\a}\a u_1}t}_{L^2} + C \norm{\en(t^{-\a}u_1)}_{L^2} \\*
		&\qquad + C\norm{\frac{t^{-\a}u_2}t}_{L^2} \\
		&\leq C\norm{t^{-\a} u_1}_{H^1_\a} + C\norm{t^{-\a-1}u_2}_{L^2},
\end{align*}
as claimed.
\end{proof}

We now have everything in place to give the proof of Theorem \ref{thm: unique continuation}:

\begin{proof}[Proof of Theorem \ref{thm: unique continuation}]
The goal is to show that $u_1$ and $u_2$ vanish on $[-T, T] \times N$ for arbitrary $T \in (0, \e)$.
Let $\varphi \in C_c^\infty((-\e, \e) \times N, \R)$, such that $\varphi = 1$ on $[-T, T] \times N$.
Define 
\[
	(f_1, f_2) \in C_*^\infty((-\e, \e) \times N, F_1 \oplus F_2)
\]
by
\[
	(f_1, f_2) =(\varphi u_1, \varphi u_2).
\]
Let $C > 0$ denote some constant which may change from term to term.
Theorem \ref{thm: Carleman estimate} and Proposition \ref{prop: Carleman estimate transport} together with Lemma \ref{le: trivial conseq from ass} imply that
\begin{align}
	\sqrt \a \norm{t^{-\a}f_1}_{H^1_\a} + &\a \norm{t^{-\a-1}f_2}_{L^2} \nonumber \\ 
		&\leq C \norm{t^{-\a}\bBox f_1}_{L^2} + \norm{t^{-\a}\on_tf_2}_{L^2} \nonumber \\
		&\leq C \norm{t^{-\a}\bBox u_1}_{L^2} + \norm{t^{-\a}\on_tu_2}_{L^2} \nonumber \\*
		&\qquad + C \norm{t^{-\a}\bBox (f_1 - u_1)}_{L^2} + \norm{t^{-\a}\on_t(f_2 - u_2)}_{L^2} \nonumber \\
		&\leq C \norm{t^{-\a} u_1}_{H^1_\a} + C \norm{t^{-\a-1}u_2}_{L^2} \nonumber \\
		&\qquad + C \norm{t^{-\a}\bBox (f_1 - u_1)}_{L^2} + \norm{t^{-\a}\on_t(f_2 - u_2)}_{L^2} \nonumber \\
		&\leq C \norm{t^{-\a} f_1}_{H^1_\a} + C \norm{t^{-\a-1}f_2}_{L^2} \nonumber \\
		&\qquad + C \norm{t^{-\a} (f_1 - u_1)}_{H^1_\a} + C \norm{t^{-\a-1}(f_2 - u_2)}_{L^2} \nonumber \\
		&\qquad + C \norm{t^{-\a}\bBox (f_1 - u_1)}_{L^2} + \norm{t^{-\a}\on_t(f_2 - u_2)}_{L^2} \label{eq: v and u}
\end{align}
for some constant $C > 0$ independent of $\a$.
We estimate the second and third terms on the right hand side in the estimate \eqref{eq: v and u} by observing
\[
	\supp\big((u_1, u_2) - (f_1, f_2)\big) \subset \big( (-\e, -T] \cup [T, \e)\big) \times N,
\]
from which we conclude that
\begin{align*}
	C \norm{t^{-\a} (f_1 - u_1)}_{H^1_\a} + C \norm{t^{-\a-1}(f_2 - u_2)}_{L^2}
		&\leq CT^{-\a}, \\
	C \norm{t^{-\a}\bBox (f_1 - u_1)}_{L^2} + \norm{t^{-\a}\on_t(f_2 - u_2)}_{L^2} 
		&\leq C T^{-\a}
\end{align*}
for some constant $C > 0$ independent of $\a$.
Note that the constant $C$ on the right hand side in these estimates depend on the functions $u_1, u_2, \varphi$, which are fixed throughout the proof.

Inserting this into the estimate \eqref{eq: v and u} implies
\begin{align*}
	\sqrt \a \norm{t^{-\a}f_1}_{H^1_\a} + \a \norm{t^{-\a-1}f_2}_{L^2} 
		&\leq C \norm{t^{-\a}f_1}_{H^1_\a} + C \norm{t^{-\a-1}f_2}_{L^2} + CT^{-\a}.
\end{align*}
For $\a$ large enough, we get the estimate
\begin{equation}
	\sqrt \a \norm{t^{-\a}f_1}_{H^1_\a} + \a \norm{t^{-\a-1}f_2}_{L^2} \leq CT^{-\a}. \label{eq: final alpha}
\end{equation}
We claim that \eqref{eq: final alpha} implies that $f_1$ and $f_2$ have to vanish on $[-T, T]\times N$.
Assume that there is a $(t_0, x_0) \in (-T, T) \times N$, such that $(f_1, f_2)(t_0, x_0) \neq 0$. 
By continuity of $f_1$ and $f_2$, there is a constant $C > 0$ (dependent on $f_1$ and $f_2$, which are fixed), such that
\[
	C \abs{t_0}^{-\a} \leq \norm{t^{-\a}f_1}_{L^2} + \norm{t^{-\a}f_2}_{L^2} \leq \sqrt \a \norm{t^{-\a}f_1}_{H^1_\a} + \a \norm{t^{-\a-1}f_2}_{L^2}.
\]
The estimate \eqref{eq: final alpha} now implies
\[
	C\abs{t_0}^{-\a} \leq CT^{-\a},
\]
or equivalently
\[
	\left(\frac{T}{\abs{t_0}}\right)^{\a} \leq C
\]
for all $\a$ large enough.
Letting $\a \to \infty$ and recalling that $T > \abs{t_0}$, we reach a contradiction.
It follows that $(f_1, f_2)(t, x) = 0$ for all $\abs{t} \leq T$. 
Consequently, $(u_1, u_2)(t, x) = 0$ for all $\abs{t} \leq T$.
Since $T \in (0, \e)$ was arbitrary, this finishes the proof.
\end{proof}

We now present the proof of Corollary \ref{cor: finite order vanishing}:
\begin{proof}[Proof of Corollary \ref{cor: finite order vanishing}]
We first prove the special case when $\kappa = \frac12$, which we have assumed throughout Section \ref{sec: unique cont}.
The idea is to use the fact that
\begin{align*}
	0 
		&= (\on_t)^m Pu|_{t = 0} \\
		&= (\on_t)^m \bBox u|_{t = 0} + (\on_t)^m B(\on u)|_{t = 0} + (\on_t)^mA u|_{t = 0},
\end{align*}
in order to prove that $\on^m u|_{t = 0} = 0$ for all $m \in \N$ and then apply Theorem \ref{thm: Wave main}.
Lemma \ref{le: comm X} implies that
\[
	[\n_t, \bBox] = g(\L_tg, \on^2) + Q_1,
\]
where $\L_t g := \L_{\d_t}g$ and $Q_1$ is a first order differential operator (with smooth coefficients).
It follows that
\begin{align}
	\left(\on_t\right)^m \bBox 
		&= \sum_{j = 0}^{m-1} \left(\on_t\right)^{m-j-1}[\on_t, \bBox]\left(\on_t\right)^j + \bBox \left(\on_t\right)^m \nonumber \\
		&= m g(\L_tg, \on^2)\left(\on_t\right)^{m-1} + \bBox \left(\on_t\right)^m + Q_m, \label{eq: m:th derivative of box}
\end{align}
where $Q_m$ is a differential operator of order $m$.
We deduce from Proposition \ref{prop: metric in time function} that
\[
	g|_{t = 0} = 	\begin{pmatrix}
							0 & 1 & 0 \\
							1 & 0 & 0 \\
							0 & 0 & \g
						\end{pmatrix} \ \Longrightarrow	\
	g^{-1}|_{t = 0} = 	\begin{pmatrix}
							0 & 1 & 0 \\
							1 & 0 & 0 \\
							0 & 0 & \g^{-1}
						\end{pmatrix}
\]
with respect to the splitting
\[
	TM|_{t = 0} = \R \d_t \oplus \R V \oplus E.
\]
Note that
\begin{align*}
	\L_tg(V, V)|_{t = 0}
		&= 2 g(\n_V \d_t, V)|_{t = 0} \\*
		&= - 2 g(\d_t, \n_V V)|_{t = 0} \\*
		&= -1.
\end{align*}
It follows that
\[
	g(\L_tg, \on^2)|_{t = 0} = - (\on_t)^2|_{t = 0} + R_1|_{t = 0},
\]
where we let $R_m$ denote some differential operator such that $R_m(u)|_{t = 0} = 0$ whenever $u|_{t = 0} = \hdots = (\on_t)^m u|_{t = 0} = 0$, i.e.\ $R_m$ differentiates in $t$-direction at most $m$ times at $t = 0$.
We allow $R_m$ to change from line to line without explicitly mentioning it.
Since $\H$ is totally geodesic, we know that $\n_XY \in T\H$ for any vector fields $X, Y$, which implies that
\begin{align*}
	\bBox|_{t = 0}
		&= -\on^2_{V, t}|_{t = 0} - R(V, \d_t)|_{t = 0} - \sum_{i,j = 2}^n \g^{ij}\on^2_{e_i, e_j}|_{t = 0} \\
		&= - 2\on_V \on_t|_{t = 0} + 2\on_{\n_V \d_t}|_{t = 0} + R_0|_{t = 0} \\
		&= - 2\on_V \on_t|_{t = 0} - \on_t|_{t = 0} + R_0|_{t = 0},
\end{align*}
where $R_0$ is as above.
Altogether, we conclude that
\begin{equation} \label{eq: time derivative box}
	(\on_t)^m\bBox|_{t = 0} = -2 \on_V (\on_t)^{m+1}|_{t = 0} - (m+1)(\on_t)^{m+1}|_{t = 0} + R_m|_{t = 0},
\end{equation}
where $R_m$ is as above.
For the first order term, we have
\begin{align}
	(\on_t)^m B(\on u)|_{t = 0}
		&= B(\on (\on_t)^mu)|_{t = 0} + R_m(u)|_{t = 0} \nonumber \\
		&= B(g(V, \cdot) \otimes (\on_t)^{m+1} u)|_{t = 0} + R_m(u)|_{t = 0}. \label{eq: time derivative B}
\end{align}
The last term is simply
\begin{equation} \label{eq: time derivative A}
	(\on_t)^mA u|_{t = 0}
		= R_m(u)|_{t = 0}.
\end{equation}
Combining equations \eqref{eq: time derivative box}, \eqref{eq: time derivative B} and \eqref{eq: time derivative A}, we observe that since $(\on_t)^m Pu|_{t = 0} = 0$, we conclude:
\begin{align*}
	2\on_V (\on_t)^{m+1} u|_{t = 0} 
		&= - (m+1)(\on_t)^{m+1}u|_{t = 0} + B(g(V, \cdot) \otimes (\on_t)^{m+1}u)|_{t = 0} \\
		&\qquad + R_m(u)|_{t = 0}.
\end{align*}
Now, by assumption in Corollary \ref{cor: finite order vanishing} we know that $(\on_t)^m u|_{t = 0}$ for all $m \leq l$.
We want to prove, by an induction argument, that $(\on_t)^m u|_{t = 0}$ for all $m \in \N_0$.
For this, assume that $(\on_t)^m u|_{t = 0} = 0$ for all $m \leq k$, where $k \geq l$.
The assumption in Corollary \ref{cor: finite order vanishing} says in particular that
\[
	\mathrm{Re}\Big(a\big(B(g(V, \cdot)\otimes (\on_t)^{k+1}u), (\on_t)^{k+1}u\big) \Big)|_{t = 0} 
		\leq l a((\on_t)^{k+1}u), (\on_t)^{k+1}u)|_{t = 0},
\]
which we use to compute that
\begin{align*}
	\d_V a((\on_t)^{k+1}u, (\on_t)^{k+1}u)|_{t = 0}
		&= \mathrm{Re}\big(a(2\on_V(\on_t)^{k+1}u, (\on_t)^{k+1}u)|_{t = 0}\big) \\
		&= -(k+1) a((\on_t)^{k+1}u, (\on_t)^{k+1}u)|_{t = 0} \\
		&\qquad +  \mathrm{Re}\big( a(B(g(V, \cdot)\otimes (\on_t)^{k+1}u), (\on_t)^{k+1}u)|_{t = 0} \big) \\
		&\leq -(k+1-l)a((\on_t)^{k+1}u, (\on_t)^{k+1}u)|_{t = 0} \\
		&\leq -a((\on_t)^{k+1}u, (\on_t)^{k+1}u)|_{t = 0}.
\end{align*}
Now, the real valued \emph{scalar function} $a((\on_t)^{k+1}u, (\on_t)^{k+1}u)|_{t = 0}$ on the \emph{compact} manifold $\H$ must attain its maximum and minimum somewhere at $\H$, say $p_{\min}$ and $p_{\max}$. 
At $p_{\min}$ and $p_{\max}$, we have
\[
	\d_V a((\on_t)^{k+1}u, (\on_t)^{k+1}u)|_{t = 0} = 0.
\]
Hence our above inequality implies that
\[
	a((\on_t)^{k+1}u, (\on_t)^{k+1}u)|_{t = 0} \leq 0,
\]
at $p_{\min}$ and $p_{\max}$.
But, since $a$ is positive definite, this implies that
\[
	a((\on_t)^{k+1}u, (\on_t)^{k+1}u)|_{t = 0} = 0
\]
at $p_{\min}$ and $p_{\max}$ and hence everywhere.
Positive definiteness of $a$ implies therefore that
\[
	(\on_t)^{k+1}u|_{t = 0} = 0.
\]
This completes the induction argument, which shows that
\[
	(\on_t)^mu|_{t = 0} = 0
\]
for all $m \in \N_0$.
The assertion, when $\kappa = \frac12$, now follows from Theorem \ref{thm: Wave main}.
If instead
\[
	\n_V V = \kappa V,
\]
for some general $\kappa$, let us apply the above to $\widetilde V := \frac{1}{2\kappa}V$, which satisfies
\[
	\n_{\widetilde V} \widetilde V = \frac12 \widetilde V.
\]
This finishes the proof.
\end{proof}

\section{Extension of Killing vector fields} \label{sec: extension of KVFs}

The purpose of this section is to apply Theorem \ref{thm: unique continuation} to prove the remaining Theorems \ref{thm: Killing Horizon main}, \ref{thm: Killing close to horizon}, \ref{thm: Killing Extension main} and \ref{thm: Killing Extension bh} and Corollary \ref{cor: second Killing}.

\subsection{Compact Cauchy horizons} \label{subsec: compact Cauchy}
Assume that $\H$ is the future Cauchy horizon of $\S$, the other case is obtained by changing the time orientation.
We begin with the proof of Theorem \ref{thm: Killing Extension main}:

\begin{proof}[Proof of Theorem \ref{thm: Killing Extension main}]

By assumption, there is a smooth vector field $Y$ such that
\begin{equation} \label{eq: vanishing Killing tensor}
	(\n_t)^m \L_Yg|_\H = 0
\end{equation}
for all $m \in \N_0$.
We begin by showing the existence of a Killing vector field $\widehat Z$ on $D(\S) \cup \H$, such that
\[
	(\n_t)^m \widehat Z|_\H = (\n_t)^m Y|_\H
\]
for all $m \in \N_0$.
Since $\Ric = 0$, equation \eqref{eq: vanishing Killing tensor} implies that
\[
	(\n_t)^m \Box Y|_\H = - (\n_t)^m\div \left(\L_Yg - \frac12 \tr_g(\L_Yg)g\right)|_\H = 0
\]
for all $m \in \N_0$.
By \cite{Petersen2018}*{Thm.\ 1.6}, there is a unique smooth vector field $\widehat Z$, defined on $D(\S) \cup \H$, such that
\begin{align*}
	\Box \widehat Z 
		&= 0, \\
	(\n_t)^m \widehat Z|_\H 
		&= (\n_t)^m Y|_\H
\end{align*}
for all $m \in \N_0$.
Inserting this into \cite{PetersenRacz2018}*{Lem.\ 2.3}, we get
\begin{equation} \label{eq: wave for tilde Z}
	\Box \L_{\widehat Z} g - 2 \mathrm{Riem}(\L_{\widehat Z}g) = 0,
\end{equation}
where $\mathrm{Riem}(\L_{\widehat Z}g)$ is a certain linear combination of $\L_{\widehat Z}g$ and the curvature tensor.
Since 
\begin{equation*}
	(\n_t)^m \L_{\widehat Z}g|_{\H} = (\n_t)^m \L_Yg|_{\H} = 0
\end{equation*}
for all $m \in \N_0$, \cite{Petersen2018}*{Thm.\ 1.6} and \eqref{eq: wave for tilde Z} imply that $\L_{\widehat Z}g = 0$ on $D(\S) \cup \H$.

We now show how to extend $\widehat Z$ to a Killing vector field $Z$ beyond $\H$.
The main ingredient in this proof is Theorem \ref{thm: unique continuation} with $N = \H$.
We combine this with Ionescu-Klainerman's recently developed method of extending Killing vector fields only based on unique continuation. 
By Theorem \ref{thm: smoothness Cauchy horizon} and \cite{Petersen2018}*{Rmk.\ 1.15}, it follows that $\H$ satisfies Assumption \ref{ass: N compact}.
Since $\widehat Z$ is a Killing vector field on $\mathcal V := (-\e, 0) \times \H \subset D(\S)$, one checks that 
\[
	\n_t \n_t \widehat Z - R(\d_t, \widehat Z)\d_t = 0
\]
on $\mathcal V$.
We define our candidate Killing vector field $Z$ on $(-\e, \e) \times \H$ by solving the linear transport equation
\begin{align}
	\n_t \n_t Z - R(\d_t, Z)\d_t 
		&= 0, \label{eq: propagation Z} \\
	Z|_{\mathcal V}
		&= \widehat Z|_{\mathcal V}. \label{eq: W and Z on O}
\end{align}

Define the smooth two-form $\o$ on $(-\e, \e) \times \H$ as the unique solution to the following linear transport equation:
\begin{align}
	\n_t \o(X, Y) 
		&= \L_Zg(X, \n_Y\d_t) - \L_Zg(\n_X\d_t, Y), \label{eq: omega transport}\\
	\o|_{\mathcal V}
		&= 0, \label{eq: vanishing one-form}
\end{align}
for any $X, Y \in T((-\e, \e) \times \H)$.
The following tensors were introduced in \cite{IonescuKlainerman2013}*{Def.\ 2.3}.
For convenience, we use abstract index notation and the Einstein summation convention:
\begin{align*}
	B_{\a \b} 
		&:= \frac12 \big((\L_Z g)_{\a\b} + \o_{\a\b} \big), \\
	P_{\a \b \gamma}
		&:= \frac12 \big(\n_\a (\L_Zg)_{\b \gamma} - \n_\b (\L_Zg)_{\a \gamma} - \n_\gamma \omega_{\a \b} \big), \\
	T_{\a \b \gamma \delta}
	 	&:= (\L_ZR)_{\a \b \gamma \delta} - {B_\a}^\lambda R_{\lambda \b \gamma \delta} - {B_\b}^\lambda R_{\a \lambda \gamma \delta} - {B_\gamma}^\lambda R_{\a \b \lambda \delta} - {B_\delta}^\lambda R_{\a \b  \gamma \lambda},
\end{align*}
where $R$ is the Riemann curvature tensor.
Ionescu and Klainerman show in \cite{IonescuKlainerman2013}*{Prop. 2.10} that these tensors satisfy a homogeneous system of linear wave equations, coupled to linear transport equations, assuming \eqref{eq: omega transport}.
In other words, there are smooth endomorphism fields $A_1, A_2$ such that
\begin{align}
	\bBox T
		&= A_1(T, \on T, B, \on B, P, \on P), \label{eq: T equation}\\
	\on_t(B, \on B, P, \on P)
		&= A_2(T, \on T, B, \on B, P, \on P), \label{eq: rest equation}
\end{align}
on $(-\e, \e) \times \H$, where $\on$ is a connection compatible with a positive definite metric on the tensors, for example a Levi-Civita connection with respect to some arbitrary choice of \emph{Riemannian} metric on $M$.

We want to apply Theorem \ref{thm: unique continuation} with
\begin{align*}
	u_1 
		&:= T, \\
	u_2
		&:= (B, \on B, P, \on P).
\end{align*}
For this, first note that \eqref{eq: W and Z on O} and \eqref{eq: vanishing one-form} imply 
\begin{align*}
	(\on_t)^m u_1|_\H
		&= 0, \\
	(\on_t)^m u_2|_\H
		&= 0
\end{align*}
for all $m \in \N_0$.
Moreover, \eqref{eq: T equation} and \eqref{eq: rest equation} can be written as
\begin{align}
	\bBox u_1
		&= A_1(u_1, \on u_1, u_2), \\
	\on_tu_2
		&= A_2(u_1, \on u_1, u_2).
\end{align}
Using the splitting \eqref{eq: tangent bundle split}, note that assumption \eqref{eq: pointwise bound assumption} in Theorem \ref{thm: unique continuation} is satisfied.
Applying Theorem \ref{thm: unique continuation} with $N = \H$, we conclude that $u_1 = 0$ and $u_2 = 0$ on $(-\e, \e) \times \H$ (after shrinking $\e$, if necessary).
It follows in particular that $B = 0$ on $(-\e, \e) \times \H$.
Since $\L_Z g$ is symmetric and $\o$ is antisymmetric, we conclude that
\[
	\L_Z g = 0
\]
on $(-\e, \e) \times \H$.
This completes the existence part of Theorem \ref{thm: Killing Extension main} with 
\[
	\O := \big((-\e, \e) \times \H\big) \cup D(\S).
\] 

For the uniqueness part, assume that $\widetilde Z$ is another Killing vector field, such that 
\[
	(\n_t)^m \widetilde Z|_\H = (\n_t)^mY|_\H = (\n_t)^m Z|_\H
\]
for any $m \in \N_0$.
On $(-\e, \e) \times \H$, we have
\begin{align*}
	\n_t \n_t \widetilde Z - R(\d_t, \widetilde Z)\d_t 
		&= 0, \\
	\widetilde Z|_\H
		&= Z|_\H, \\
	\n_t \widetilde Z|_\H
		&= \n_t Z|_\H.
\end{align*}
But since $Z$ also solves this linear transport equation with the same initial data, it follows that $Z = \widetilde Z$ on $(-\e, \e) \times \H$. 
Moreover, since $\Ric = 0$ and $\L_{\widetilde Z}g = \L_Z g = 0$, we have
\[
	\Box (Z - \widetilde Z)|_{D(\S)} = - \div \left(\L_{Z - \widetilde Z}g - \frac12 \tr_g(\L_{Z - \widetilde Z}g) g \right)|_{D(\S)} = 0.
\]
Therefore, standard theory or \cite{Petersen2018}*{Thm.\ 1.6} implies that $Z|_{D(\S)} = \widetilde Z|_{D(\S)}$ as well.
This completes the proof. 
\end{proof}

\begin{proof}[Proof of Theorem \ref{thm: Killing Horizon main} and Theorem \ref{thm: Killing close to horizon}]
By \cite{PetersenRacz2018}*{Thm.\ 1.2} and \cite{PetersenRacz2018}*{Rmk.\ 3.1}, there is a unique Killing vector field $W$ on $D(\S) \cup \H$, which satisfies
\begin{align*}
	W|_\H 
		&= V, \\
	[\d_t, W]
		&= 0, \quad \text{ on } (-\e, 0] \times \H.
\end{align*}
Consistent with Theorem \ref{thm: Killing close to horizon}, we may therefore extend the Killing vector field to $(-\e, \e) \times \H$ by solving $[\d_t, W] = 0$.
It remains to prove that $\L_Wg = 0$ also on $(0, \e) \times \H$, i.e.\ beyond the Cauchy horizon.
Since we know that
\[
	(\n_t)^m \L_Wg|_{\H} = 0
\]
for all $m \in \N_0$, Theorem \ref{thm: Killing Extension main} implies the existence of a unique Killing vector field $Z$ on $(-\e, \e) \times \H$ such that 
\[
	(\n_t)^m Z|_\H = (\n_t)^m W|_\H,
\]
for all $m \in \N_0$.
Note that it suffices to prove that $W = Z$ beyond the Cauchy horizon, i.e.\ on $(0, \e) \times \H$.
Recall the defining equation \eqref{eq: propagation Z} in the proof of Theorem \ref{thm: Killing Extension main}, which says that $Z$ satisfies the transport equation
\[
	\n_t \n_t Z - R(\d_t, Z)\d_t = 0
\]
on $(-\e, \e) \times \H$.
Since we know that $Z$ and $W$ coincide on $(-\e, 0] \times \H$, it suffices to prove that
\[
	\n_t \n_t W - R(\d_t, W)\d_t = 0
\]
on $(-\e, \e) \times \H$, by uniqueness of linear transport equations.
Using that $[W, \d_t] = 0$ and $\n_t \d_t = 0$, we compute
\begin{align*}
	\n_t \n_t W - R(\d_t, W)\d_t
		&= \n_t\n_W\d_t - R(\d_t, W)\d_t \\
		&= \n_W \n_t \d_t \\
		&= 0,
\end{align*}
as claimed. 
Therefore $W = Z$ on $(-\e, \e) \times \H$ and we conclude that $\L_Wg = 0$ on $(-\e, \e) \times \H$.
By Proposition \ref{prop: metric in time function} and Corollary \ref{cor: N is Cauchy horizon}, we know that $W$ is spacelike in $D(\S)$ close to $\H$, lightlike on $\H$ and timelike in $\O \backslash (D(\S) \cup \H)$ close to $\H$. 

For the uniqueness part, assume that $\widetilde W$ is a second Killing vector field on $\O$ such that
\[
	\widetilde W|_\H = V.
\]
We claim that $\widetilde W = W$.
Since $W$ and $\widetilde W$ are Killing vector fields such that $\widetilde W|_\H = V = W|_\H$, we get
\begin{align*}
	g(\n_t \widetilde W, \d_t)|_\H
		&= 0 = g(\n_t W, \d_t)|_\H, \\
	g(\n_t \widetilde W, X)|_\H
		&= - g(\n_X \widetilde W, \d_t)|_\H = - g(\n_X W, \d_t)|_\H = g(\n_t W, X)|_\H
\end{align*}
for any $X \in T\H$.
It follows that $\n_t \widetilde W|_\H = \n_t W|_\H$.
Since both $W$ and $\widetilde W$ are Killing vector fields, we know that
\[
	\n_t\n_t (\widetilde W - W) = R(\d_t, \widetilde W - W)\d_t
\]
on $(-\e, \e) \times \H$.
Hence
\[
	\widetilde W|_{(-\e, \e) \times \H} = W|_{(-\e, \e) \times \H}.
\]
The uniqueness part of Theorem \ref{thm: Killing Extension main} implies therefore that $\widetilde W = W$ on $\O$, as claimed.
\end{proof}

\begin{proof}[Proof of Corollary \ref{cor: second Killing}]
By \cite{PetersenRacz2018}*{Cor.\ 1.1} and its proof, there is a second Killing vector field $Z$ on $D(\S)$ such that $[\d_t, Z] = 0$ up to $\H$.
In particular, it follows that
\[
	(\n_t)^m\L_Zg|_{t = 0} = 0
\]
for all $m \in \N_0$.
By Theorem \ref{thm: Killing Extension main}, there is an extension of $Z$ beyond $\H$.
Applying the same method as in the proof of Theorem \ref{thm: Killing Horizon main}, one notes that $[\d_t, Z] = 0$ also beyond the horizon. 
The uniqueness part of Theorem \ref{thm: Killing Extension main} implies that $Z$ is different from $W$. 
Moreover, if $[Z, W] = 0$ on one side of the horizon, then we have
\[
	[\d_t, [Z, W]] = - [W, [\d_t, Z]] - [Z, [W, \d_t]] = 0,
\]
proving that $[Z, W] = 0$ also on the other side of the horizon.
By \cite{PetersenRacz2018}*{Cor.\ 1.1}, this proves the last assertion.
\end{proof}

\subsection{Black hole event horizons} \label{subsec: event horizon}

We now prove Theorem \ref{thm: Killing Extension bh}, using Theorem \ref{thm: Killing Horizon main}.
Recall that we denoted the \emph{stationary} Killing vector field on $M$ by $K$.
By \cite{ChruscielCosta2008}*{Prop.\ 4.1 \& Prop.\ 4.3 \& Thm.\ 4.11}, $\H_{bh} \subset M$ is a smooth hypersurface and there is a smooth hypersurface $S_0 \subset \H_{bh}$ (codimension $2$ submanifold in $M$) which is transversal to both $K$ and to all generators (lightlike integral curves) in $\H_{bh}$.
Moreover, all integral curves of $K$ along $\H_{bh}$ and all generators of $\H_{bh}$ intersect $S_0$ precisely once.
In addition, we have this:

\begin{lemma} \label{le: event tot geodesic}
The event horizon $\H_{bh}$ is totally geodesic.
\end{lemma}
This statement is classical, but in lack of an appropriate reference, let us give the proof here:
\begin{proof}
The proof is a straightforward modification of the first part of the proof of Proposition \ref{prop: null time function}.
As in that proof, one constructs the vector bundle 
\[
	T\H / {\R V}
\]
and defines the expansion $\theta$, which satisfies $\theta \leq 0$.
Since $[V, K] = 0$ and $K$ is a Killing vector field, it follows that $\d_K \theta = 0$. 
Since $V$ is nowhere vanishing on $\H_{bh}$ and each integral curve of $V$ intersects $S_0$, the flow generated by $V$ is an diffeomorphism of $\H_{bh}$ without fix points, commuting with the flow of $K$.
Thus, since $K$ is nowhere vanishing at $S_0$, $K$ does not vanish anywhere on $\H_{bh}$ and any integral curve of $K$ intersects $S_0$ exactly once.
Since $K$ is complete, tangent to $\H_{bh}$ and transversal to $S_0$, we may flow $S_0$ along $K$ and write $\H_{bh}$ as a foliation $\H_{bh} = \R \times S_0$, where $\{0\} \times S_0$ corresponds to $S_0$.
In particular, the flow of $K$ induces a free and proper $\R$-action of isometries.

As in the proof of Proposition \ref{prop: null time function}, one constructs the Riemannian metric $\sigma$ on $\H_{bh}$, such that
\begin{equation} \label{eq: flow V sigma}
	\L_Vd\mu_\sigma = -\theta d \mu_\sigma,
\end{equation}
where $d\mu_\sigma$ is the volume density with respect to $\sigma$.
The construction of $\sigma$ relied on \cite{Larsson2015}*{Lem.\ 1.3}, in which there is a certain freedom in choosing $\sigma$.
We will use this freedom here and make a suitable choice.
Fix a unit timelike vector field $T$ at $S_0$ (which is necessarily transversal to $\H_{bh}$) such that $g(V, T) = 1$.
Then extend $T$ along $\H_{bh}$ by requiring that $[T, K] = 0$.
It follows that $g(T, T) = -1$ and that $g(V, T) = 1$ on $\H_{bh}$, in particular $T$ is transversal to $\H_{bh}$.
Now, \cite{Larsson2015}*{Lem.\ 1.3} implies that the Riemannian metric
\[
	\sigma(X, Y) := g(X, Y) + g(X, T) g(Y, T)
\]
satisfies \eqref{eq: flow V sigma}.
Note, moreover, that $\L_K \sigma = 0$, which implies that 
\begin{equation} \label{eq: flow K sigma}
	\L_K d\mu_\sigma = 0.
\end{equation}

Since the flow $\phi_s$ of $K$ ($s \in \R$ is the flow time) gives a free and proper $\R$-action of isometries, we may pass to the quotient $\H_{bh}/{\sim}$, where $p {\sim} q$ if there is a $z \in \Z$ such that $p = \phi_z(q)$.
Since $[K, V] = 0$ and $\d_K \theta = 0$, $V$ and $\theta$ descend to the quotient and equation \eqref{eq: flow V sigma} now holds on the \emph{compact} manifold $\H_{bh}/{\sim}$.
Flowing along $V$ on $\H_{bh}/{\sim}$ induces a diffeomorphism of a compact manifold.
This means that the volume is finite and has to stay the same.
Since we already know that $\theta \leq 0$, equation \eqref{eq: flow V sigma} implies that $\theta = 0$ on $\H_{bh}/{\sim}$.
Therefore $\theta = 0$ also on $\H_{bh}$.
Proceeding as in Proposition \ref{prop: null time function}, we conclude that $\H_{bh}$ is totally geodesic.
\end{proof}

By substituting $V$ by $\frac1{2\kappa}V$ we may assume that $\kappa = \frac12$. 
Since $\H_{bh}$ is a future event horizon, the integral curves of $V$ are future complete.
This implies that $V$ is future directed.
Along the lines in Subsection \ref{subsec: null time function}, using that $\H_{bh}$ is totally geodesic, one shows that there is a unique \emph{past directed} lightlike vector field $L$ along $\H_{bh}$ (c.f.\ Figure \ref{fig: EH to CH}) such that
\begin{align*}
	g(L, V)
		&= 1, \\
	g(L, X)
		&= 0,	
\end{align*}
for all $X \in T\H_{bh}$ such that $\n_X V = 0$.
Note that $L$ is nowhere vanishing and transveral to $\H_{bh}$.
Let $\phi_s$ denote the flow under $K$.
Since $\phi_s$ are isometries and $d \phi_s(V) = V$, we conclude that $\n_{d\phi_s(X)} V = 0$ for all vectors $X \in T\H_{bh}$ such that $\n_X V = 0$.
Using this, note that
\begin{align*}
	g(d\phi_s(L), V)
		&= 1, \\
	g(d\phi_s(L), X)
		&= 0,	
\end{align*}
for all $X \in T\H_{bh}$ such that $\n_X V = 0$.
Consequently, $d\phi_s(L) = L$, i.e.\ 
\begin{equation}
	[K, L] = 0.
\end{equation}
Our strategy is to construct a null time function in a neighbourhood of the event horizon by first flowing $S_0$ along the geodesics in direction of $L$ and then apply the flow of $K$.
The fact that $[K, L] = 0$ plays a crucial role in this argument.

\begin{lemma}[The null time function for event horizons]
The smooth map
\begin{align*}
	F: (-\e, \e) \times \R \times S_0 
		&\to M \\
	(t, s, p) 
		&\mapsto \exp(tL|_{\phi_s(p)}),
\end{align*}
is an immersion for some $\e > 0$ and there is an open subset
\[
	\left([0, \e) \times \R \times S_0\right) \subset \mathcal V \subset \left((-\e, \e) \times \R \times S_0 \right),
\]
such that $F|_\mathcal V$ is a diffeomorphism onto its image.
Moreover:
\begin{itemize}
	\item $F\left((-\e, 0) \times \R \times S_0\right) \subset \mathcal B$, i.e.\ the black hole region,
	\item $F\left(\{0\} \times \R \times S_0\right) = \H_{bh}$, i.e.\ the event horizon,
	\item $F\left((0, \e) \times \R \times S_0\right) \subset \langle\langle M_\ext \rangle\rangle$, i.e.\ the domain of outer communication.
\end{itemize}
\end{lemma}

The level sets of null time function are illustrated in Figure \ref{fig: EH to CH}.

\begin{remark}
Let us emphasise that it is not clear whether one can find an $\e > 0$ such that $F$ is injective.
The problem is that the subset $(-\e, 0) \times \R \times S_0$ is mapped into the black hole region, where we have made essentially no regularity assumptions (c.f.\ Assumption \ref{ass: black hole}). 
However, one can clearly modify our assumptions on stationary black holes slightly in order to make sure that $F$ is injective.
Exactly this technical point is the reason why we can only assure that $\U \cap \left( \H_{bh} \cup \langle\langle M_\ext \rangle\rangle\right)$ is invariant under $K$ in Theorem \ref{thm: Killing Extension bh}.
From our assumptions, we cannot assure that $\U$ is invariant under $K$ in the black hole region.
\end{remark}

\begin{proof}

Note first that since $\phi_s$ are isometries, we have
\begin{align*}
	\exp(tL_{\phi_s(p)})
		&= \exp(d\phi_s(tL_p)) \\
		&= \phi_s(\exp(tL_p)).
\end{align*}
Hence $F$ can be written as the composition $F = \phi_s \circ f$, where 
\begin{align*}
	f: (-\e, \e) \times S_0
		&\to M \\
	(t, p)
		&\mapsto \exp(tL|_p).
\end{align*}
By compactness of $S_0$, the smooth map $f$ is well-defined and a diffeomorphism onto its image, for some small $\e > 0$.
Shrinking $\e$ if necessary, we can make sure that $K$ is nowhere vanishing on $\im(f)$.
Since $\phi_s$ is a diffeomorphism, it follows that $F$ is an immersion. 
By construction, 
\[
	F(0, s, p) = \phi_s(p)
\]
for all $(s, p) \in \R \times S_0$, which proves that $F|_{\{0\} \times \R \times S_0}$ is a diffeomorphism onto $\H_{bh}$.

The existence of \emph{some} open set $\mathcal V$ such that $F|_\mathcal V$ is injective is therefore clear.
Let us now argue why we may conclude that $\mathcal V\supset \left([0, \e) \times \R \times S_0\right)$.
First of all, since $L$ is past directed, we know that $dF(\d_t)$ is past directed.
Since $\H_{bh}$ is a future event horizon, we conclude that $dF(\d_t)$ points into the domain of outer communication.
The statement now follows by the structure theorem for stationary black holes, \cite{ChruscielCosta2008}*{Thm.\ 4.5}.
\end{proof}

\begin{remark}
Even though we do not know that $F$ is injective, we may pull back the metric $g$ along $F$ and consider the Lorentzian manifold $((-\e, \e) \times \R \times S_0, F^*g)$.
We use the coordinates $(t, s)$ on $(-\e, \e) \times \R$.
One readily notes the following properties:
\begin{itemize}
	\item it is a vacuum spacetime, 
	\item the vector field $\d_s = F^*K$ is a Killing vector field,
	\item the hypersurface $\{0\} \times \R \times S_0$ is a lightlike hypersurface with constant non-zero surface gravity.
\end{itemize}
\end{remark}

\begin{proof}[Proof of Theorem \ref{thm: Killing Extension bh}]
Let $\varphi_s$ denote the flow of $\d_s$ on $(-\e, \e) \times \R \times S_0$.
Note that $\phi_s = F\circ \varphi_s$.
We get an \emph{isometric} action by $\Z$, given by
\begin{align*}
	\Z \times \left((-\e, \e) \times \R \times S_0 \right) 
		&\to (-\e, \e) \times \R \times S_0, \\
	(z, (t, s, p)) 
		&\mapsto (t, s+ z, p).
\end{align*}
Using this action, we pass to the locally isometric quotient
\[
	((-\e, \e) \times S^1 \times S_0, F^*g),
\]
which is a smooth vacuum spacetime.
Now, the hypersurface $N := \{0\} \times S^1 \times S_0$ is a \emph{compact} lightlike hypersurface, which is totally geodesic by Lemma \ref{le: event tot geodesic}.
Recall from Definition \ref{def: constant surface gravity event} that
\[
	[K, V] = 0
\]
on $\H_{bh}$, and hence
\[
	[\d_s, F^*V] = 0,
\]
which implies that $F^*V$ descends as a lightlike vector field on N.
Moreover, since the surface gravity of the event horizon was a non-zero constant, the same is true for $N$.
The projection 
\[
	t: (-\e, \e) \times S^1 \times S_0 \to (-\e, \e)
\]
is the null time function in the sense of Proposition \ref{prop: null time function}.
Corollary \ref{cor: N is Cauchy horizon} and Theorem \ref{thm: Killing Horizon main} imply therefore that there is a unique Killing vector field $W$ such that $[\d_t, W] = 0$ and $W|_{t = 0} = F^*V$.
It follows that 
\[
	[\d_s, W]|_{t = 0} = [\d_s, F^*V] = 0.
\]
Moreover, we have 
\[
	[\d_t, [\d_s, W]] = [[\d_t, \d_s], W] + [\d_s, [\d_t, W]] = 0.
\]
It follows that $[\d_s, W] = 0$ on $(-\e, \e) \times S^1 \times S_0$.
We may therefore lift the Killing vector field $W$ to $(-\e, \e) \times \R \times S_0$, still denoting it $W$.
Since $F|_\mathcal V$ is a diffeomorphism, we may push forward $W$ to a Killing vector field on $\U := F(\mathcal V)$, which we again call $W$, such that
\[
	W|_{\H_{bh}} = V
\]
and $[K, W] = 0 = [\d_t, W]$.
This completes the proof.
\end{proof}

\section*{Acknowledgements}
The author would like to thank Vincent Moncrief for important discussions and the Priority Program 2026: Geometry at Infinity, funded by Deutsche Forschungsgemeinschaft, for financial support.

\end{sloppypar}
\end{document}